\documentclass[]{amsart}
\usepackage[utf8]{inputenc}
\usepackage{amsmath}
\usepackage{amssymb}
\usepackage{amsthm}
\usepackage{tikz}
\usepackage{xcolor}
\usepackage{hyperref}
\usepackage{url}
\usepackage{stmaryrd}

\hypersetup{
	colorlinks=true,       
	linkcolor=blue,          
	citecolor=blue,        
	filecolor=blue,      
	urlcolor=blue           
}

\numberwithin{equation}{section}
\numberwithin{figure}{section}

\allowdisplaybreaks
\theoremstyle{plain}
\newtheorem{lemma}{Lemma}[section]
\newtheorem{proposition}[lemma]{Proposition}
\newtheorem{corollary}[lemma]{Corollary}

\newtheorem{theorem}[lemma]{Theorem}

\theoremstyle{definition}
\newtheorem{definition}[lemma]{Definition}
\newtheorem{remark}[lemma]{Remark}

\newtheorem{example}[lemma]{Example}

\newcommand{\R}{\mathbb{R}}
\newcommand{\N}{\mathbb{N}}
\newcommand{\diff}{\,\textnormal{d}}
\newcommand{\ind}{\mathbf{1}}
\renewcommand{\P}{\mathbb{P}}
\newcommand{\F}{\mathcal{F}}
\newcommand{\E}{\mathbb{E}}
\newcommand{\obj}{M}
\newcommand{\val}{V}
\newcommand{\Pcal}{\mathcal{P}}
\newcommand{\weak}{\tau_{\textnormal{w}}}
\newcommand{\dvechaus}{\vec{d}}
\newcommand{\cpt}{\mathrm{K}}
\DeclareMathOperator{\skewsym}{A}
\DeclareMathOperator{\ortho}{SO}
\newcommand{\hilbert}{\mathbb{H}}
\newcommand{\covspace}{\mathbb{K}}
\DeclareMathOperator{\trace}{tr}

\begin{document}
	
	\title{Fr\'echet Means in Infinite Dimensions}
	
	\author{Adam Quinn Jaffe}
	\address{Department of Statistics\\
		Columbia University\\ 
		New York, NY 10027 \\
		U.S.A.}
	\email{a.q.jaffe@columbia.edu}
	
	
	
	\date{\today}
	
	\begin{abstract}
		While there exists a well-developed asymptotic theory of Fr\'echet means of random variables taking values in a general ``finite-dimensional'' metric space, there are only a few known results in which the random variables can take values in an ``infinite-dimensional'' metric space.
		Presently, we develop a general asymptotic theory of Fr\'echet means in some infinite-dimensional metric spaces, which allows us to recover, strengthen, and generalize most existing results; in particular, we develop novel asymptotic theory for Fr\'echet means in some infinite-dimensional metric spaces from statistical shape analysis.
		The core of the proof is a novel notion of weak convergence in general metric spaces for which the results can be proven via calculus of variations.
	\end{abstract}
	
	\maketitle

	\tableofcontents

	\section{Introduction}\label{sec:intro}
	
	\subsection{Problem Statement} For a metric space $(X,d)$, a real number $p \ge 1$, and a finite set of points $Y_1,Y_2,\ldots, Y_n\in X$, we consider the optimization problem:
	\begin{equation}\label{eqn:obj-ex}
		\begin{cases}
			\textrm{minimize } & \frac{1}{n}\sum_{i=1}^{n}d^p(x,Y_i) \\ 
			\textrm{ over } & x\in X.
		\end{cases}
	\end{equation}
	This problem should be familiar when $X=\R$ and $d$ is its usual metric.
	If $p=2$ then the solution is unique and given by the mean of $Y_1,\ldots, Y_n$, and, if $p=1$, then solutions to \eqref{eqn:obj-ex} are exactly the medians of $Y_1,\ldots, Y_n$, although uniqueness need not hold in general.
	
	As such, many authors, especially in the field of \textit{non-Euclidean statistics,} regard a solution of \eqref{eqn:obj-ex} as a canonical notion of central tendency in a general metric space, although this object goes by many different names in the literature.
	For general $p$, they are most often called \textit{Fr\'echet $p$-means} \cite{ManifoldsI, ManifoldsII, EvansJaffeFrechet, Schoetz2, FrechetPersistence} or \textit{$L^p$ centers of mass} \cite{Afsari, Karcher, Jost, Kendall}.
	In the special case of $p=2$ these may be simply abbreviated to \textit{Fr\'echet means} or \textit{centers of mass}, although it is also common to refer to them as \textit{barycenters} \cite{AGP, SturmNPC, Picard, WassersteinBarycenter, WassersteinConvergence, AguehCarlier}.
	In the $p=1$ case they can be referred to as either \textit{medians} or \textit{geometric medians} \cite{TurnerMedians, WassersteinMedian}, although some authors use the term \textit{Fermat-Weber point} \cite{FermatWeberRevisited, ChandrasekaranTamir} in reference to the \textit{Fermat-Weber problem} or the closely-related \textit{facility location problem}.
	There is also a natural extension of this theory to the $p=\infty$ case (which we will not discuss in this paper), sometimes called \textit{circumcenters}, \textit{Chebyshev centers}, or otherwise  \cite{EvansLidman, CuestaMatran, Lember}.
	In this work we prefer to use the term Fr\'echet $p$-mean, but usually $p$ will be understood from context and hence omitted.
	
	In analogy with the usual limit theorems for averages of random variables in Euclidean spaces, it is natural to hope that there is a nice asymptotic theory for solutions to problem \eqref{eqn:obj-ex}, under suitable probabilistic assumptions on the underlying random variables $Y_1,Y_2,\ldots$ and under suitable geometric assumptions on the underlying metric space $(X,d)$.
	To state such results precisely, we need to introduce some notation.
	Since the problem \eqref{eqn:obj-ex} depends on $Y_1,\ldots, Y_n$ only through the empirical measure $\bar\mu_n:=\frac{1}{n}\sum_{i=1}^{n}\delta_{Y_i}$, it makes sense to define
	\begin{equation}\label{eqn:obj-def}
		\obj_p(\mu) := \underset{x\in X}{\arg\min}\int_{X}d^p(x,y)\diff\mu(y),
	\end{equation}
	for $\mu$ any Borel probability measure $\mu$ on $(X,d)$ with $p$ finite moments; we take this to be the entire set of minimizers, which can be empty, a singleton, or contain multiple points.
	In fact, one can relax the condition of $p$ finite moments to a condition of $p-1$ finite moments by using $|d^p(x, y)-d^p(o,y)|\le pd(x,o)(d^{p-1}(x,y)+d^{p-1}(o,y))$ for all $o,x,y,\in X$, which holds in any metric space; then for fixed $o\in X$, we define the set
	\begin{equation}\label{eqn:obj-renorm}
		\obj_{p}(\mu;o) := \underset{x\in X}{\arg\min}\int_{X}(d^p(x,y)-d^p(o,y))\diff\mu(y).
	\end{equation}
	Note that $\obj_{p}(\mu;o)$ does not depend on $o\in X$, and that we have $\obj_p(\mu) = \obj_p(\mu;o)$ whenever $\mu$ has $p$ finite moments, so this is in some sense a canonical extension of $M_p$.
	For the sake of exposition, let us focus for the moment on the question ofwhen  the strong law of large numbers (SLLN) holds, meaning $M_p(\bar \mu_n)$ converges almost surely to $M_p(\mu)$ when $Y_1,Y_2,\ldots$ are identically-distributed (IID) samples from distribution $\mu$ on some probability space; later we will address the case of further limit theorems of interest.
	
	Many authors have formulated versions of the SLLN for Fr\'echet means \cite{Ziezold, SverdrupThygeson, EvansJaffeFrechet, Schoetz2, Huckemann, ManifoldsI, ManifoldsII} under suitable assumptions; the strongest general result to date is the recent work of Sch\"otz \cite{Schoetz2} in which the following is shown: If $(X,d)$ is a Heine-Borel metric space, meaning all closed, bounded sets are compact, and if $Y_1,Y_2,\ldots$ is an IID sequence of $X$-valued random variables with common distribution $\mu$ possessing $p-1$ finite moments, then we have
	\begin{equation}\label{eqn:SLLN}
		\max_{\bar x_n\in \obj_p(\bar \mu_n)}\min_{x\in \obj_p(\mu)}d(\bar x_n,x) \to 0
	\end{equation}
	almost surely.
	We regard this as a sort of uniform guarantee ``no false positives'' in the sense of statistical testing, since it shows that all points in $\obj_p(\bar \mu_n)$ are close to some point of $\obj_p(\mu)$, asymptotically as $n\to\infty$.
	
	The primary limitation of the work of Sch\"otz above is that, since a Heine-Borel metric space is essentially ``finite-dimensional'', the result does not directly establish any asymptotic theory for Fr\'echet means in various ``infinite-dimensional'' metric spaces of interest non-Euclidean statistics.
	In some particular cases (e.g., the Wasserstein space of probability measures \cite{WassersteinBarycenter}, the space of persistence diagrams \cite{DivolLacombe}, and metric spaces of non-positive curvature \cite{SturmNPC}), asymptotic theory of Fre\'chet means has been established by other methods.
	But for some metric spaces of interest in statistical shape analysis \cite{MichorMumford, Kurtek, Sundaramoorthi, BruverisMichorMumford, TumpachPreston}, no asymptotic theory for Fr\'echet means is known at all, despite their ubiquitous role in various statistical and computational procedures.
	Thus, the goal of this work is to develop a sufficiently general asymptotic theory of Fr\'echet means in abstract metric spaces which unifies these previous examples and extends to additional examples of practical interest in applications.
	
	A secondary goal of this work is to establish this asymptotic theory under the minimal moment condition of $p-1$ finite moments.
	Indeed, most works on Fr\'echet means do not employ the renormalization~\eqref{eqn:obj-renorm} as in Sch\"otz \cite{Schoetz2}, and hence they are forced to focus on the unnecessarily strong assumption of $p$ finite moments \cite{Huckemann, ManifoldsI, ManifoldsII, WassersteinBarycenter, DivolLacombe, Ziezold, SverdrupThygeson, Austin, Thorpe, PanaretosSantoroBW}; exceptions include \cite{SturmNPC} where some additional convexity structure is present.
	Thus, even in cases where the SLLN is already known (e.g., \cite{WassersteinBarycenter,DivolLacombe}), our goal is to provide this improvement on the moment conditions.
	Note that in general one cannot dispense with $p-1$ finite moments for convergence of the Fr\'echet $p$-mean, since in the case of the real line $X=\R$ with its usual metric and $p=2$ we know that the SLLN fails without finite expectation.
	
	\subsection{Summary and Outline}
	
	In this subsection we explain how this work achieves both of the goals set forth above.
	The subsection is divided into three parts which each describe the main results of a section in the body of the paper.
	
	\subsubsection*{Weak Convergence (Section~\ref{sec:weak})}
	
	A careful inspection of some existing results on the SLLN for Fr\'echet means \cite{Schoetz2, WassersteinBarycenter,PersistenceUniqueness} reveals that most proofs are based on the direct method of calculus of variations.
	To state this more precisely, let us assume that $Y_1,Y_2,\ldots$ are IID samples from a distribution $\mu$ on a metric space $(X,d)$, and let us take some $\{x_n\}_{n\in\N}$ in $X$ with $x_n\in M_p(\mu_n)$ for all $n\in\N$.
	Each of the authors above, in their respective metric space of interest, identifies a suitable notion of ``weak convergence'' such that the following steps can be executed:
	\begin{itemize}
		\item[(S0)] The sequence $\{x_n\}_{n\in\N}$ is $d$-bounded.
		\item[(S1)] There exists $\{n_k\}_{k\in\N}$ and $x$ in $X$ with $x_{n_k}\to x$ weakly.
		\item[(S2)] The point $x$ lies in $\obj_p(\mu)$.
		\item[(S3)] We in fact have $x_{n_k}\to x$ in $d$.
	\end{itemize}
	The first part of the paper is dedicated to a careful analysis of a particular notion of weak convergence in abstract metric spaces, which is amenable to the analysis above.
	
	Roughly speaking, we say (Definition~\ref{def:weak-like}) that a \textit{weak convergence} for a metric space $(X,d)$ is a convergence satisfying
	\begin{itemize}
		\item[(W1)] $d$-bounded sets are relatively weakly compact.
		\item[(W2)] $d$ is weakly lower-semicontinuous.
		\item[(W3)] $d$-convergence is weak convergence plus convergence of distance to a point.
	\end{itemize}
	We note that weak convergences are not required to correspond to convergence in a bona fide topology on $(X,d)$ and that, in case there is such a topology, it may be neither metrizable nor second countable.
	
	Our results on weak convergences are twofold.
	First (Subsection~\ref{subsec:ex}, summarized in Table~\ref{tab:weak-like}), we show that many metric spaces encountered non-Euclidean statistics admit a weak convergence in the sense described above; these include abstract Heine-Borel metric spaces, uniformly convex Banach spaces, Hadamard spaces, and the Wasserstein space.
	Second (Subsection~\ref{subsec:closure}, summarized in Table~\ref{tab:weak-closure}), we show that the class of metric spaces admitting a weak convergence is closed under many operations including restriction to subspaces, finite products, quotients by compact groups acting by isometries, and deformation by proper metric groups acting by isometries.
	These two collections of results, when combined, provide a calculus for establishing the existence of a weak convergence for further metric spaces encountered in practice, and we use this technique to analyze many examples (Subsection~\ref{subsec:further}) including the rotationally-invariant Wasserstein space, spaces of images deformed by rotations, a Sobolev metric on the the space of unparameterized unregistered planar loops, the space of persistence diagrams endowed with the partial matching metric, and the Bures-Wasserstein metric on a space of infinite-dimensional covariance operators.
	
	\subsubsection*{The Main Theorem (Section~\ref{sec:main})}
	
	The bulk of the paper is spent proving our main technical result, that the Fr\'echet mean is a continuous function of its underlying probability measure, whenever the metric space $(X,d)$ is separable and admits a weak convergence.
	
	To state this precisely, let us introduce some notation.
	For $r\ge 0$, we write $\Pcal_r(X)$ for the space of all Borel probability measures $\mu$ on $(X,d)$ such that for some (equivalently, all) $x\in X$ we have  $\int_{X}d^{r}(x,y)\diff \mu(y)<\infty$ and we write $\weak^r$ for the topology on $\Pcal_r(X)$ such that $\{\mu_n\}_{n\in\N}$ and $\mu$ in $\Pcal_r(X)$ have $\mu_n\to \mu$ if and only if we have $\int_{X}f\diff \mu_n\to \int_{X}f\diff\mu$ for all bounded continuous $f:(X,d)\to \R$ and $\int_{X}d^{r}(x,y)\diff \mu_n(y)\to \int_{X}d^{r}(x,y)\diff \mu(y)$ for some (equivalently, all) $x\in X$ as $n\to\infty$.
	
	Our main result (Theorem~\ref{thm:main}) concerns fixed $p\ge 1$ and a separable metric space $(X,d)$ admitting a weak convergence, and it states the following:
	If $\{\mu_n\}_{n\in\N}$ and $\mu$ are any probability measures in $\mathcal{P}_{p-1}(X)$ satisfying $\mu_n\to\mu$ in $\weak^{p-1}$, then every sequence $\{x_n\}_{n\in\N}$ with $x_n\in M_p(\mu_n)$ for each $n\in\N$ must be relatively $d$-compact and any subsequential limit thereof must lie in $M_p(\mu)$.
	This is the so-called $\Gamma$-convergence of $M_p(\mu_n)$ to $M_p(\mu)$, but it also has several more interpretable consequences.
	For example, it implies (Corollary~\ref{cor:main-cpt}) that $M_p(\mu)$ is a non-empty $d$-compact set for each $\mu\in\Pcal_{p-1}(X)$, and also (Corollary~\ref{cor:main-cty}) that $\mu_n\to\mu$ in $\weak^{p-1}$ implies
	\begin{equation}\label{eqn:main-cvg}
		\max_{x_n\in \obj_p(\mu_n)}\min_{x\in \obj_p(\mu)}d(x_n,x) \to 0
	\end{equation}
	as $n\to\infty$.
	
	Two remarks about the main theorem are due.
	First, we emphasize that the weak convergence appears in this result only in the hypotheses; the Fr\'echet means are computed with respect to the metric $d$, and all compactness and convergence is shown to hold with respect to the metric $d$;
	in this sense, we say that the existence of a weak convergence is actually a geometric condition on $(X,d)$ under which its Fr\'echet means are well-behaved.
	Second, we emphasize that this result is purely analytic; while we will later specialize to the case that $\{\mu_n\}_{n\in\N}$ are the empirical measures of some data points, we have actually shown a more general statement about Fr\'echet means of any convergent sequence of measures, and this generality will be exploited later.
	
	The proof of this main theorem, as discussed above, is precisely to execute the steps (S0), (S1), (S2), and (S3) above.
	In fact, we show that step (S0) holds in an arbitrary metric space, and each that step (S$i$) is proven with the help of condition (W$i$), for each $i\in\{1,2,3\}$.
	As such, one can dispense with condition (W3) if only non-emptiness and weak convergence of the empirical Fr\'echet means is desired (Remark~\ref{rem:thm-without-W3}), although we find this to be an insufficient conclusion for most applications.
	
	\subsubsection*{Probabilistic Consequences (Section~\ref{sec:prob})}
	
	The final part of the paper shows, plainly, that the main theorem yields a great deal asymptotic theory for Fr\'echet means.
	Throughout this part, suppose that $(X,d)$ is a metric space admitting a weak convergence.
	
	The simplest and most fundamental result (Corollary~\ref{cor:SLLN}) is the strong law of large numbers (SLLN) which states that, if $Y_1,Y_2,\ldots$ is an independent, identically-distributed (IID) sequence of $X$-valued random variables satisfying $\E[d^{p-1}(x,Y_1)]<\infty$ for some (equivalently, all) $x\in X$, then we have
	\begin{equation*}
		\max_{\bar x_n\in \obj_p(\bar \mu_n)}\min_{x\in \obj_p(\mu)}d(\bar x_n,x) \to 0
	\end{equation*}
	almost surely, where $\bar \mu_n:=\frac{1}{n}\sum_{i=1}^{n}\delta_{Y_i}$ is the empirical measure of the first $n\in\N$ points; this shows the desired SLLN, whereby $p-1$ finite moments suffices for the convergence of $M_p$.
	For $p=2$ this proves that 1 finite moment suffices for the convergence of $M_2$, and for $p=1$ this proves that no integrability at all is required in order to guarantee the convergence of $M_1$.
	
	The SLLN provides basic consistency for Fr\'echet means used in applications of non-Euclidean statistics, and our result has several consequences that we now explain.
	First, there are many applied and computational uses of Fr\'echet means in statistical shape analysis  \cite{Kurtek, Sundaramoorthi, SrivastavaKlassen,MillerYounes}, although no SLLN was previously known in the infinite-dimensional examples of interest; our results thus provide novel guarantees for the rotationally-invariant Wasserstein space \cite{WassersteinTomography,WassersteinAlignment,ProcrustesWasserstein}, spaces of images deformed by rotations \cite{MillerYounes,ShapeSpaces}, and the Sobolev metric on the space of unparameterized unregistered planar loops \cite{MichorMumford, TumpachPreston, BruverisMichorMumford}.
	Second, there are many applications where the SLLN was already known but where our result sharpens the moment condition to the minimal possible, e.g., uniformly convex Banach spaces \cite{CuestaMatran}, the Wasserstein space \cite{WassersteinBarycenter}, the space of persistence diagrams endowed with the partial matching metric \cite{DivolLacombe}, and the Bures-Wasserstein metric on a space of infinite-dimensional covariance operators \cite{PanaretosSantoroBW}.
	Lastly, there are some cases where the SLLN was already known under the minimal condition, and our results recovers these (in addition to establishing new asymptotic theory below), e.g., Heine-Borel spaces \cite{Schoetz2} and Hadamard spaces \cite{SturmNPC}.
	
	More sophisticated probabilistic results involve almost sure convergence for much wider probabilistic settings.
	For instance, we have (Corollary~\ref{cor:ergodic}) a pointwise ergodic theorem for certain amenable groups actions.
	That is, suppose that $G$ is a locally compact amenable group acting in a measure-preserving way on a probability space $(\Omega,\F,\P)$; if $Y:\Omega\to X$ is any map whose law is in $\Pcal_{p-1}(X)$, then we can show that the Fr\'echet mean of the empirical measures of $Y$, averaged along its orbits truncated via a suitable F{\o}lner sequence, converge almost everywhere to an invariant random set.
	Along the same lines, one can prove almost sure convergence for sequences of data which come from exchangeable sequences, Harris-recurrent Markov chains, 1-dependent sequences, and more.
	
	In the setting of ergodic theorems, the minimality of our moment assumptions deserves further discussion.
	The pointwise ergodic theorem for Fr\'echet 2-means (barycenters) in Hadamard spaces was initially studied in \cite{Austin} under the sub-optimal assumption of 2 finite moments, and it was asked whether one can prove the same result under the assumption of only 1 finite moment (even in the setting of $G=\mathbb{Z}$).
	The work \cite{Navas} (see also \cite{EsSahibHeinich}) provides one solution to this problem, but not regarding the Fr\'echet mean in \eqref{eqn:obj-renorm}; rather, it introduces an alternative notion which is more amenable to the analysis of \cite{Austin}.
	In contrast, our work answers this question affirmatively for the canonical Fr\'echet mean in \eqref{eqn:obj-renorm}, for all $p$ and for all metric spaces admitting a weak convergence.
	
	Another probabilistic result (Corollary~\ref{cor:LDP}) establishes the large deviations principle (LDP); that is, if $Y_1,Y_2,\ldots$ is an IID sequence of $X$-valued random variables satisfying $\E[\exp(\lambda d^{p-1}(x,Y_1))]<\infty$ for all $\lambda>0$ and some (equivalently, all) $x\in X$, and assuming that $M_p(\bar \mu_n)$ is unique almost surely for each $n\in\N$, we have
	\begin{align*}
		-\inf\{I_{p,\mu}(x): x\in A^{\circ}\}
		&\le \liminf_{n\to\infty}\frac{1}{n}\log \P(M_p(\bar \mu_n)\in A) \\
		&\le \limsup_{n\to\infty}\frac{1}{n}\log \P(M_p(\bar \mu_n)\in A) \le -\inf\{I_{p,\mu}(x): x\in \bar A\}
	\end{align*}
	for all measurable $A\subseteq X$ and for a suitable function $I_{p,\mu}:X\to \R$.
	In some sense, the LDP provides a general-purpose method for analyzing the (asymptotic) concentration of measure phenomenon for Fr\'echet means.
	This extends the results in \cite{JaffeSantoroLDP}, although the rate function $I_{p,\mu}$ above is given in the form of a relative entropy projection while \cite{JaffeSantoroLDP} develops a more concrete form of the LDP using further tools from convex analysis.
	
	\subsection{Related Literature}
	
	There is a wide, and somewhat disparate, literature on asymptotic theory for Fr\'echet means, so we take some time to explain the literature in order to appropriately situate our results.
	
	The story of the SLLN begins with Ziezold \cite{Ziezold} who established a consistency result in a general separable metric space under $p$ finite moments which states that we have
	\begin{equation*}
		\bigcap_{n\in\N}\overline{\bigcup_{m\ge n}M_p(\bar\mu_m)} \subseteq M_p(\mu)
	\end{equation*}
	almost surely; this notion of convergence is sometimes now referred to as \textit{Ziezold consistency}.
	While this result is extremely general, the notion of convergence is quite weak; it only guarantees that, if $\{M_p(\bar\mu_n)\}_{n\in\N}$ has any accumulation point, then it must lie in $M_p(\mu)$, but it does not guarantee that any accumulation point exists.
	Parallel work by Sverdrup-Thygesen \cite{SverdrupThygeson} restricts to compact metric spaces and shows
	\begin{equation*}
		\max_{\bar x_n\in \obj_p(\bar \mu_n)}\min_{x\in \obj_p(\mu)}d(\bar x_n,x) \to 0
	\end{equation*}
	almost surely, which is the same notion of convergence we study in Corollary~\ref{cor:main-cty}; this notion of convergence was popularized in the special case of Riemannian manifolds in the work of Bhattacharya and Patrangenaru \cite{ManifoldsI,ManifoldsII} and is now sometimes referred to as \textit{Bhattacharya-Patrangenaru consistency}.
	The recent work \cite{EvansJaffeFrechet} extended Ziezold's result to the case of $p-1$ finite moments in a general separable metric space, and the work \cite{Schoetz2} extended Sverdrup-Thygesen's result to the case of $p-1$ finite moments in a general Heine-Borel metric space; strictly speaking, \cite{Schoetz2} gives the result under a weaker ``sample-Heine-Borel'' property, but explains that establishing sample-Heine-Borel in a metric space which is not itself Heine-Borel ``may be of similar difficulty as showing convergence of Fr\'echet means directly''.
	
	While the previous discussion was limited to the case of the SLLN for Fr\'echet means in abstract metric spaces, there is a parallel literature on the SLLN in particular metric spaces of interest.
	Examples include Riemannian manifolds \cite{ManifoldsI}, stratified spaces \cite{Huckemann, UnlabeledGraphs, LeKume}, some Banach spaces \cite{Lember,CuestaMatran}, Hadamard spaces \cite{SturmNPC}, the Wasserstein and Bures-Wasserstein spaces \cite{WassersteinBarycenter, PanaretosSantoroBW}, and spaces of persistence diagrams \cite{DivolLacombe}.
	We emphasize that these results are proven using a variety of methods, but that our main theorem allows us to recover all of them on a clean basis, and in all cases to sharpen the requisite moment condition.
	
	Another branch of literature studies SLLN for generalizations of Fr\'echet means, in a few different directions.
	For example, \cite{Schoetz2} in fact proves the SLLN for minimizers of $n^{-1}\sum_{i=1}^{n}\tau(d(x,Y_i))$ for a general class of functions $\tau:[0,\infty)\to [0,\infty)$ including our $\tau(u)=u^p$ for $p\ge 1$.
	Another example is that $d$ does not need a bona fide metric on a set $X$, but rather can be any loss function between two spaces, provided that certain axioms are satisfied; consistency results in this setting are taken up in \cite{Huckemann2015} with motivation from geodesic PCA.
	We focused on the case of Fr\'echet $p$-means for simplicity in this work so our results do not directly imply anything about these generalized settings, but we expect our methods will indeed be useful in such problems in the future.
	
	We regard the calculus of Section~\ref{subsec:closure} to be in the spirit of Grenander's famed pattern theory \cite{MumfordPattern, MumfordPatternTheory}, since it gives easy-to-verify conditions under which one can build new metric spaces from simpler ones, in which the general asymptotic theory holds.
	As illustrated in Example~\ref{ex:loops} on the Sobolev metric on unregistered unparameterized planar loops, this is particular interesting from the point of view of statistical shape analysis, since there is a huge variety of metrics of interest \cite{SrivastavaKlassen, Kurtek, Sundaramoorthi, MichorMumford} each with their own advantages and disadvantages, from a pratical point of view.
	
	We also emphasize that our results make no assumptions about the uniqueness of Fr\'echet means.
	While previous literature has often restricted attention to the case of a unique Fr\'echet mean \cite{ManifoldsI, ManifoldsII, FrechetRegression, DepthProfile}, recent works have attempted to go beyond this; some authors provide simple assumptions under which one can deduce that the Fr\'echet mean is unique \cite{Afsari, FrechetCircle, PersistenceUniqueness, SturmNPC, LeKume}, and others embrace the generality of the set-valued setting \cite{EvansJaffeFrechet, Schoetz2, Huckemann, WassersteinBarycenter}.
	The results herein, which are stronger than all known results even in the uniqueness setting, are completely valid in the set-valued setting, and they provide a sort of uniform convergence over all elements of the corresponding Fr\'echet mean sets.
	
	The question of whether our results recover some form of central limit theorem (CLT) for Fr\'echet means is more subtle.
	The CLT has indeed been established in some metric spaces, like Riemannian manifolds \cite{ManifoldsII}, some stratified spaces \cite{BookCLTs}, and the Wasserstein and Bures-Wasserstein spaces \cite{Kroshnin, agueh2017vers}.
	However, even the statement of such results requires some differentiable structure; the population Fr\'echet mean is supposed to admit a linear tangent space, and the CLT guarantees convergence in distribution of the rescaled empirical Fr\'echet mean (to a suitable Gaussian measure) in this tangent space.
	In recent work \cite{TreeCLT,MMT_I,MMT_II,MMT_III,MMT_IV}, it has been established that qualitatively different behavior may occur when the requisite tangent space degenerates.
	Thus, it not possible to state (let alone prove) a general CLT for Fr\'echet means without further geometric restriction.
	Nevertheless, we emphasize that most CLTs are proven by first using an SLLN to show that the empirical Fr\'echet mean is close to the population Fr\'echet mean, and second by applying some Taylor expansion to the Fr\'echet functional in a suitable sense; thus, our results still help to complete the first step of this program.
	
	Our results also do not recover existing theory on finite-sample rates of convergence for Fr\'echet means and related objects; indeed, this is why we draw the distinction between ``asymptotic'' and ``non-asymptotic'' results.
	For example, explicit rates of convergence are known in Hadamard spaces \cite{Schoetz1} and Wasserstein and Bures-Wasserstein spaces \cite{AGP,WassersteinConvergence,Kroshnin}, but these results rely on much stronger inputs, both in terms of the geometry of the underlying metric space and in terms of the assumptions on the population distribution and its Fr\'echet mean.

	As inspiration for future work, let us mention a few settings in which we believe it would be tractable and interesting to verify the existence of a weak convergence.
	First is the setting of Banach manifolds modeled on uniformly convex Banach spaces; this would be especially interesting because such spaces are fundamental to theoretical framework of shape analysis \cite{BruverisMichorMumford}.
	Second is a general setting of metrics defined via a ``dynamical formulation'', including the metric arising in large deformation diffeomorphic metric mapping (LDDMM) \cite{YounesShapes}, the Wasserstein-Fisher-Rao metric \cite{WFR}, and more; notice that we already established such a result in our study of the Wasserstein space, but that we only used the static formulation in the present work.
	Third is the particular example of the invariant $L^2$-metric on the space of graphons \cite[Section~14.2]{Lovasz}, which arises as the natural continuum limit of the metric which has previously been studied in non-Euclidean approaches to statistics on spaces of unlabeled graphs \cite{UnlabeledGraphs}; our existing  result on quotient metric spaces (Proposition~\ref{prop:quotient}) does not directly apply in this case, since we now have a quotient of an infinite-dimensional metric space by an infinite-dimensional isometric group action.
	
	Lastly, let us mention that our notion of weak convergence in Definition~\ref{def:weak-like} may be useful in other applications of analysis in metric spaces, beyond the theory of Fr\'echet means.
	For example, similar (but slightly different) conditions have been used in the theory of gradient flows in metric spaces \cite[Section~2.1]{AmbrosioGigliSavare} and in the study of fixed-point theorems in metric spaces \cite{KirkPanyanak, EspinolaFernandezLeon, Lim}.
	
	\section{Weak Convergence}\label{sec:weak}
	
	This section contains a careful study of a notion of weak convergence in metric spaces, as this will be the main analytic tool used in proving our main results.
	More specifically, we introduce this notion of weak convergence (Subsection~\ref{subsec:def}), we give some basic examples (Subsection~\ref{subsec:ex}), we establish some useful closure properties (Subsection~\ref{subsec:closure}), and then we give a few further examples of a more complicated nature (Subsection~\ref{subsec:further}).
	
	\subsection{Definitions}\label{subsec:def}
	
	First we introduce some precise notions that we will need throughout the section.
	By a \textit{convergence} $c$ on $X$ we simply mean a way to assign at most one value of $X$ to each sequence in $X$, with the additional property that this value is preserved under taking subsequences; if a convergence $c$ assigns value $x$ to a sequence $\{x_n\}_{n\in\N}$, then we write \textit{$x_n\to x$ in $c$}.
	(Note that some authors may refer to this as a \textit{Hausdorff} or $T_2$ convergence, since we are, roughly speaking, requiring that limits are uniquely-defined.)
	
	Specifically, we are interested in convergences with the following structure:
	
	\begin{definition}\label{def:weak-like}
		For a metric space $(X,d)$, a convergence $w$ on $X$ is called a \textit{weak convergence for $(X,d)$} if it satisfies the following conditions:
		\begin{enumerate}
			\item[(W1)] If $\{x_n\}_{n\in\N}$ and $y$ in $X$ satisfy $\sup_{n\in\N}d(x_n,y)<\infty$, then there exists a subsequence $\{n_k\}_{k\in\N}$ and a point $x\in X$ satisfying $x_{n_k}\to x$ in $w$.
			\item[(W2)] If $\{x_n\}_{n\in\N}$ and $x$ in $X$ satisfy $x_n\to x$ in $w$, then for all $y\in X$ we have $d(x,y)\le \liminf_{n\in\N}d(x_n,y)$.
			\item[(W3)] If $\{x_n\}_{n\in\N}$ and $x$ in $X$ satisfy $x_n\to x$ in $w$ and satisfy $d(x_n,y)\to d(x,y)$ for some $y\in X$, then we have $x_n\to x$ in $d$.
		\end{enumerate}
		We say that $(X,d)$ \textit{admits a weak convergence} if there exists a weak convergence $w$ for $(X,d)$.
	\end{definition}
	
	The upshot of this work is that the existence of weak convergence is actually a geometric property of a metric space which ensures that its Fr\'echet means are well-behaved.
	We note that conditions similar to (but slightly different than) the above have appeared in previous works on gradient flows in metric spaces \cite[Section~2.1]{AmbrosioGigliSavare} and $k$-means clustering in certain Banach spaces \cite{Thorpe}.
	Before we get into the details of the main results, let us give some remarks about Definition~\ref{def:weak-like}.
	Throughout these remarks, let $(X,d)$ denote a metric space and $w$ a weak convergence for $(X,d)$.
	
	First, we note that the term ``weak'' is meant to be suggestive of the fact that there exist many examples of weak convergences connected to the usual notions of weak topology in functional analysis and in probability.
	While we will explain this in detail in Subsection~\ref{subsec:ex}, the reader can see Table~\ref{tab:weak-like} for some concrete examples: Heine-Borel metric spaces with their metric convergence, uniformly convex Banach spaces with their weak topology, Hadamard spaces with Jost's weak convergence, and the Wasserstein space with the topology of weak convergence of measures.
	
	Second, we note that aspects of topologies related to $w$ can be subtle.
	For example, if $(X,d)$ is a Hadamard space and $w$ is Jost's weak convergence, then the existence of a topology on $X$ corresponding to convergence in $w$ depends intricately on whether one considers bounded or unbounded sets, sequences or nets, and more \cite{LytchakPetrunin, WeakUnbddNet}.
	Another example the simple case where $(X,d)$ is a uniformly convex Banach space and $w$ is its usual weak convergence, in which case there exists a topology corresponding to convergence in $w$ but this topology is neither metrizable nor second-countable.
	Because our probabilistic motivations lead us to focus only on sequences, we will not pursue in this work the question of whether a weak convergence necessarily corresponds to convergence in some topology; however, we believe that such questions would be interesting for future work.
	
	Third, notice that we did not explicitly require in Definition~\ref{def:weak-like} that $w$ should be ``weaker than'' $d$ in any sense.
	This is indeed true, and we now show that it follows from the other assumptions.
	However, notice that, since $w$ may not correspond to convergence in a topology, the statement takes on a slightly different form than expected.
	
	\begin{lemma}
		If $w$ is a weak convergence for a metric space $(X,d)$ and if $\{x_n\}_{n\in\N}$ and $x$ in $X$ have $x_n\to x$ in $d$, then every subsequence of $\{x_n\}_{n\in\N}$ admits a further subsequence converging in $w$ to $x$.
	\end{lemma}
	
	\begin{proof}
		Let $\{n_k\}_{k\in\N}$ be an arbitrary subsequence.
		By assumption, the sequence $\{x_{n_k}\}_{k\in\N}$ converges in $d$ and is hence $d$-bounded.
		Thus, (W1) guarantees the existence of $\{k_j\}_{j\in\N}$ and $x'\in X$ such that we have $x_{n_{k_j}}\to x'$ in $w$.
		Now observe
		\begin{equation*}
			d(x',x) \le \liminf_{j\to\infty}d(x_{n_{k_j}},x) =  0
		\end{equation*}
		by (W2), which implies $x=x'$.
		We have thus shown $x_{n_{k_j}}\to x$ in $w$, as needed.
	\end{proof}

	\subsection{Basic Examples}\label{subsec:ex}
	
	Next, we give a few example of metric spaces for which we can directly verify the existence of a weak convergence, either by-hand or by referencing known results in the literature.
	A summary of these results can be found in Table~\ref{tab:weak-like}.

	\begin{table}
		\tiny
		\centering
		\begin{tabular}{|c|c|c|c|c|}
			\hline
			\shortstack{\\metric space\\\,} & \shortstack{\\weak convergence\\\,} & \shortstack{\\(W1)\\\,} & \shortstack{\\(W2)\\\,} & \shortstack{\\(W3)\\\,} \\
			\hline
			\textbf{\shortstack{Heine-Borel space\\ \,}} & \shortstack{\\metric convergence\\\,} & \shortstack{\\definition\\\,} & \shortstack{\\trivial\\\,}& \shortstack{\\trivial\\\,} \\
			\hline
			\textbf{\shortstack{uniformly convex\\Banach space}} & \shortstack{weak topology\\\,} & \shortstack{\\Milman-Pettis\\theorem} & \shortstack{weak lower semi-\\continuity of norm} & \shortstack{Kadec-Klee\\ property} \\
			\hline
			\textbf{\shortstack{Hadamard space\\ \,}} & \shortstack{Jost's convergence\\\,} & \shortstack{\\Jost's Banach-\\Alaoglu theorem} & \shortstack{weak lower semi-\\continuity of metric} & \shortstack{Kadec-Klee\\ property} \\
			\hline
			\textbf{\shortstack{Wasserstein space\\ \,}} & \shortstack{\\weak convergence of\\probability measures} & \shortstack{Prokhorov\\theorem} & \shortstack{Fatou lemma\\\,} & \shortstack{definition\\\,} \\
			\hline
		\end{tabular}
		\bigskip
		\caption{The basic examples of weak convergences.
			While the full details are given Subsection~\ref{subsec:ex}, we observe here that the axioms of weak convergence correspond, roughly speaking, to well-known notions from functional analysis, metric geometry, and probability theory.}
		\label{tab:weak-like}
	\end{table}

	The first setting is Heine-Borel metric spaces which, we emphasize, is the entire scope of most previous works in the abstract setting.
	
	\begin{example}[Heine-Borel space]\label{ex:HB}
		Suppose $(X,d)$ has the property that the closed balls $\bar B_r(x):=\{y\in X: d(x,y)\le r\}$ are compact for all $x\in X$ and $r\ge 0$.
		Then, it is easy to see that convergence with respect to $d$ is itself a weak-convergence for $(X,d)$.
		(In fact, the metric convergence is a weak convergence if and only if $(X,d)$ is a Heine-Borel metric space.)
	\end{example}
	
	Second is the setting of uniformly convex Banach spaces, in which Fr\'echet means have been studied from the point of view of functional data analysis, approximation theory, and $k$-means clustering.
	We note that many classical examples of Banach spaces are uniformly convex.
	For example, all $L^q$ spaces with $1<q<\infty$ over a $\sigma$-finite measure space \cite[Theorem~5.2.11]{Megginson} and all Sobolev spaces $H^{k,q}(D)$ for $k\ge 0$, $1<q<\infty$, and a bounded open domain $D\subseteq\R^m$ \cite[Proposition~8.1]{Brezis}.
	Additionally, there are some known conditions on an Orlicz function $\Phi:[0,\infty)\to[0,\infty)$ which imply that the Orlicz space $L_{\Phi}$ is uniformly convex \cite{Kaminska}.

	\begin{example}[uniformly convex Banach space]\label{ex:unif-cvx-Banach}
		Suppose $(X,\|\cdot\|)$ is a uniformly convex Banach space with $d$ the metric induced by the norm, and let $w$ denote the usual notion of weak convergence in $(X,\|\cdot\|)$.
		We claim $w$ is indeed a weak convergence for $(X,d)$ in the sense of Definition~\ref{def:weak-like}.
		To see (W1), recall by the Milman-Pettis theorem \cite[Theorem~5.2.15]{Megginson} that $(X,\|\cdot\|)$ must be reflexive and it is well-known that a Banach space is reflexive if and only if its norm-closed balls are weakly compact.
		For (W2), simply recall that norms are always weakly lower semi-continuous.
		Lastly, note that (W3) is called the \textit{Kadec-Klee property}, the \textit{Radon-Riesz property}, or \textit{property (H)} in this Banach space setting, and it is known \cite[Theorem~5.2.18]{Megginson} that uniformly convex Banach spaces have this property.
	\end{example}
	
	Third is the setting of metric spaces whose curvature (in the sense of Alexandrov) is everywhere non-positive.
	Such spaces have been an important element of non-Euclidean statistics since \cite{SturmNPC}, and examples include Hilbert spaces, complete simply-connected Riemannian manifolds of non-positive sectional curvature, metric trees, the so-called \textit{Billera-Holmes-Vogtmann (BHV) treespace} studied in computational phylogenetics, and more \cite{SturmNPC, BacakOptimization, BHV}.
	
	\begin{example}[Hadamard space]\label{ex:hadamard}
		Suppose $(X,d)$ is a (complete) geodesic metric space with the property that all geodesic triangles are ``at least as thin as'' the corresponding geodesic triangles in the Euclidean plane; such spaces are called \textit{Hadamard spaces}, \textit{non-positively curved (NPC) spaces}, or \textit{CAT(0) spaces}.
		In particular, we emphasize that $(X,d)$ is not assumed to be locally compact.
		
		The notion of weak convergence in Hadamard spaces is usually attributed to Jost in \cite[Definition~2.7]{JostWeak}, but it is known to be closely related to earlier work on metric fixed-point theory \cite{KirkPanyanak, Lim}; see \cite[Proposition~3.1.3]{BacakOptimization}.
		We define it as follows:
		Let us say that $\{x_n\}_{n\in\N}$ and $x$ in $X$ satisfies $x_n\to x$ in $w$ if for every geodesic ray $\gamma$ emanating from $x$, we have $p_{\gamma}(x_n)\to p_{\gamma}(x)$ as $n\to\infty$, where $p_{\gamma}:X\to\gamma$ denotes the metric projection onto $\gamma$.
		
		We can collect some existing results in the literature to see that $w$ is indeed a weak convergence in the sense of Definition~\ref{def:weak-like}:
		Condition (W1) is proven in \cite[Theorem~2.1]{JostWeak} and \cite[Proposition~3.1.2]{BacakOptimization} which is Jost's generalization of the Banach-Alaoglu theorem to Hadamard spaces.
		Moreover, condition (W2) is proven in \cite[Corollary~3.2.4]{BacakOptimization}, and condition (W3) is proven in \cite[Proposition~3.1.6]{BacakOptimization}.
	\end{example}

	More generally than CAT(0) spaces, one has a notion of CAT($\kappa$) spaces for $\kappa\in\R$, which are, roughly speaking, metric spaces with curvature bounded above by $\kappa$.
	Some known results suggest that Jost's convergence (or an analog of it) may provide a suitable notion of weak convergence in CAT($\kappa$) spaces of sufficiently small diameter \cite{EspinolaFernandezLeon, Kell}, but more work is needed in order to carefully verify the conditions (W1), (W2), and (W3) above.
	
	Lastly, we consider the setting of the Wasserstein space in which Fr\'echet means (usually called \textit{barycenters}) have attracted much recent attention \cite{WassersteinBarycenter, AguehCarlier, PanaretosSantoroBW, JaffeSantoroLDP}.
	
	\begin{example}[Wasserstein space]\label{ex:wass}
		Let $(S,\rho)$ denote a complete, separable, and locally compact metric space.
		For $1\le q<\infty$ let us write $\Pcal_q(S)$ for the space of probability measures $P$ on $S$ with $\int_{S}\rho^q(s,t)\diff P(s)<\infty$ for some (equivalently, all) $t\in S$, and let us write $W_q$ for the $q$-Wasserstein metric on $\Pcal_q(S)$.
		Our interest lies in the metric space $(\Pcal_q(S),W_q)$.
		
		We claim that the topology of weak convergence of measures, denoted $\weak$ gives rise to a weak convergence for $(\Pcal_q(S),W_q)$ in the sense of Definition~\ref{def:weak-like}.
		Recall that $\{P_n\}_{n\in\N}$ and $P$ in $\Pcal_q(S)$ are said to satisfy $P_n\to P$ in $\weak$ when we have $\int_{S} f \diff P_n\to \int_{S} f \diff P$ for all bounded continuous functions $f:(S,\rho)\to \R$.
		
		Indeed, let us show that conditions (W1), (W2), and (W3) are satisfied.
		For (W1), suppose that $\{P_n\}_{n\in\N}$ and $Q$ in $W_q(S)$ satisfy
		\begin{equation}\label{eqn:wass-1}
			\sup_{n\in\N}W_q(P_n,Q)<\infty,
		\end{equation}
		and use the local compactness of $S$ to get a sequence $K_1\subseteq K_2\subseteq \cdots$ of compact subsets of $S$ satisfying $\bigcup_{\ell\in\N}K_{\ell} = S$.
		Now fix $t\in K_1$, and note that \eqref{eqn:wass-1} and the triangle inequality imply
		\begin{equation*}
			C:=\sup_{n\in\N}\int_{S}\rho^{q}(s,t)\diff P_n(s) < \infty.
		\end{equation*}
		Also, for each $\ell\in\N$ we have
		\begin{equation*}
			\ind_{S\setminus K_{\ell}} \le \frac{\rho(\,\cdot\,,t)}{\inf_{s\notin K_{\ell}}\rho(s,t)},
		\end{equation*}
		so taking the $q$th power and integrating with respect to $P_n$ gives
		\begin{equation}\label{eqn:wass-2}
			P_n(S\setminus K_{\ell}) \le \frac{C}{\inf_{s\notin K_{\ell}}\rho^{q}(s,t)}.
		\end{equation}
		Since we have $\inf_{s\notin K_{\ell}}\rho(s,t) \to \infty$ as $\ell\to\infty$, equation~\eqref{eqn:wass-2} shows that $\{P_n\}_{n\in\N}$ is tight.
		Thus, the Prokhorov theorem shows that $\{P_n\}_{n\in\N}$ admits a subsequence converging in $\weak$, and this is exactly (W1).
		Also note that (W2) is proven in \cite[Proposition~7.1.3]{AmbrosioGigliSavare}, which essentially consists of Fatou's lemma and an approximation argument.
		Lastly, note that (W3) is proven in \cite[Lemma~14]{WassersteinBarycenter}; this is similar to the definition of Wasserstein convergence (weak convergence plus convergence of $q$th moment), although the proof is a bit more complicated.
	\end{example}
	
	\subsection{Closure Properties}\label{subsec:closure}
	
	We now give various closure properties for the class of metric spaces admitting a weak convergence.
	A summary of these results can be found in Table~\ref{tab:weak-closure}.
	
	\begin{table}
		\tiny
		\centering
		\begin{tabular}{|c|c|c|c|}
			\hline
			\shortstack{\\metric space\\operation} & \shortstack{\\inputs\\\,}  & \shortstack{\\output\\\,} & \shortstack{\\weak convergence\\\,} \\
			\hline
			\textbf{\shortstack{restriction\\\,}} & \shortstack{\\$(X,d)$ with weak convergence $w$,\\$w$-closed subset $X'\subseteq X$\\\,} & \shortstack{$(X',d)$\\\,} & \shortstack{$w$\\\,} \\
			\hline
			\textbf{\shortstack{product\\\,}} & \shortstack{\\$(X_1,d_1)$ with weak convergence $w_1$,\\$(X_2,d_2)$ with weak convergence $w_2$\\\,} & \shortstack{$(X_1\times X_2,d_1\otimes_{q}d_2)$\\\,} & \shortstack{$w_1\otimes w_2$\\\,} \\
			\hline
			\textbf{\shortstack{\\quotient\\\,}} & \shortstack{\\$(X,d)$ with weak convergence $w$,\\compact group $G$ acting isometrically\\\,} & \shortstack{$(X/G,d_{X/G})$\\\,} & \shortstack{\\$\Gamma$-convergence of\\orbits under $w$\\\,} \\
			\hline
			\textbf{\shortstack{\\deformation\\\,}} & \shortstack{\\$(X,d)$ with weak convergence $w$,\\proper metric group $(G,\rho)$ acting isometrically\,}& \shortstack{\\$(X,d_{G,\rho})$\\\,} & \shortstack{\\$w$\\\,} \\
			\hline
		\end{tabular}
		\bigskip
		\caption{The closure properties for weak convergences.
			While the full details are given Subsection~\ref{subsec:closure}, we observe here that many operations on metric spaces preserve the existence of a weak convergence.}		
		\label{tab:weak-closure}
	\end{table}
	
	The first result concerns subspace restriction, and involves an additional definition.
	If $c$ is a convergence on a set $X$, let us say that a subset $X'\subseteq X$ is \textit{$c$-closed} if $\{x_n\}_{n\in\N}$ in $X'$ and $x\in X$ satisfying $x_n\to x$ in $c$ implies $x\in X'$.
	Note that this terminology is slightly misleading in two senses: First, $c$ need not correspond to convergence in a topology; second, we are requiring that $X'$ contain all of the limits of sequences in $X'$ but not necessarily of nets in $X'$.
	Nonetheless, we believe this abuse of terminology will cause no confusion.
	
	\begin{proposition}\label{prop:restriction}
		Consider any metric space $(X,d)$ admitting a weak convergence $w$, and any $w$-closed subset $X'\subseteq X$.
		Then, $w$ is a weak convergence for $(X',d)$ .
	\end{proposition}
	
	\begin{proof}
		We only need to show that (W1) holds for $w$ on $(X',d)$.
		Indeed, if $\{x_n\}_{n\in\N}$ is $d$-bounded, then (W1) for $w$ on $(X,d)$ implies that there exists $\{n_k\}_{k\in\N}$ and $x\in X$ such that we have $x_{n_k}\to x$ in $w$.
		But $X'\subseteq X$ being $w$-closed implies that we must have $x\in X'$, as needed.
	\end{proof}
	
	The next result concerns product metrics, and requires some notation.
	For two metric spaces $(X_1,d_1),(X_2,d_2)$ and some $1\le q<\infty$, we define the metric $d_1\otimes_q d_2$ on $X_1\times X_2$ via
	\begin{equation*}
		(d_1\otimes_q d_2)((x_1,x_2),(x_1',x_2')) :=(d_1^q(x_1,x_1') + d_2^q(x_2,x_2'))^{1/q}.
	\end{equation*}
	Also, for two convergences $c_1,c_2$ on $X_1,X_2$, respectively, we define $c_1\otimes c_2$ to be the convergence on $X_1\times X_2$ such that $\{(x_{1,n},x_{2,n})\}_{n\in\N}$ and $(x_1,x_2)$ in $X_1\times X_2$ satisfy $(x_{1,n},x_{2,n})\to (x_1,x_2)$ in $c_1\otimes c_2$ if and only if we have both $x_{1,n}\to x_1$ in $c_1$ and $x_{2,n}\to x_2$ in $c_2$.
	Then we have the following:
	
	\begin{proposition}\label{prop:product}
		Consider any metric spaces $(X_1,d_1), (X_2,d_2)$ admitting weak convergences $w_1,w_2$, respectively, and let $1\le q<\infty$ be arbitrary.
		Then, $w_1\otimes w_2$ is a weak convergence for $(X_1\times X_2,d_1\otimes_q d_2)$.
	\end{proposition}
	
	\begin{proof}
		First we consider (W1), so we suppose that $\{(x_{1,n},x_{2,n})\}_{n\in\N}$ and $(y_1,y_2)$ in $X_1\times X_2$ satisfy
		\begin{equation*}
			\sup_{n\in\N} (d_1\otimes_q d_2)((x_{1,n},x_{2,n}),(y_1,y_2)) < \infty.
		\end{equation*}
		By construction this implies that we have
		\begin{equation*}
			\sup_{n\in\N} d_1(x_{1,n},y_1) < \infty
		\end{equation*}
		and
		\begin{equation*}
			\sup_{n\in\N} d_2(x_{2,n},y_2) < \infty,
		\end{equation*}
		so by applying (W1) for both $w_1$ and $w_2$, we can get a single subsequence $\{n_k\}_{k\in\N}$ and some $x_1\in X_1$ and $x_2\in X_2$ satisfying both $x_{1,n_k}\to x_1$ in $w_1$ and $x_{2,n_k}\to x_2$ in $w_2$ as $k\to\infty$.
		By the construction of $w$, this means $(x_{1,n},x_{2,n})\to (x_1,x_2)$ in $w_1\otimes w_2$, as needed.
		
		Next we consider (W2), so we suppose that $\{(x_{1,n},x_{2,n})\}_{n\in\N}$ and $(x_1,x_2)$ in $X_1\times X_2$ satisfy  $(x_{1,n},x_{2,n})\to (x_1,x_2)$ in $w_1\otimes w_2$.
		By using (W2) for both $w_1$ and $w_2$, we have
		\begin{align*}
			(d_1\otimes_q d_2)((x_{1},x_{2}),(y_1,y_2)) &= (d_1^q(x_1,y_1) + d_2^q(x_2,y_2))^{1/q} \\
			&\le (\liminf_{n\to\infty}d_1^q(x_{1,n},y_1) + \liminf_{n\to\infty}d_2^q(x_{2,n},y_2))^{1/q} \\
			&\le (\liminf_{n\to\infty}(d_1^q(x_{1,n},y_1) +d_2^q(x_{2,n},y_2)))^{1/q} \\
			&= \liminf_{n\to\infty}(d_1^q(x_{1,n},y_1) +d_2^q(x_{2,n},y_2))^{1/q} \\
			&= \liminf_{n\to\infty}(d_1\otimes_q d_2)((x_{1,n},x_{2,n}),(y_1,y_2))
		\end{align*}
		by continuity. This shows (W2) for $w_1\otimes w_2$, as claimed.
		
		Lastly,  we show (W3).
		To do this, suppose that $\{(x_{1,n},x_{2,n})\}_{n\in\N}$ and $(x_1,x_2)$ in $X_1\times X_2$ satisfy  $(x_{1,n},x_{2,n})\to (x_1,x_2)$ in $w_1\otimes w_2$, and suppose that $(y_1,y_2)\in X_1\times X_2$ satisfies
		\begin{equation}\label{eqn:product-1}
			(d_1\otimes_q d_2)((x_{1,n},x_{2,n}),(y_1,y_2)) \to (d_1\otimes_q d_2)((x_{1},x_{2}),(y_1,y_2))
		\end{equation}
		as $n\to\infty$.
		Since (W2) for $w_2$ implies
		\begin{equation}\label{eqn:product-2}
			d_2(x_{2},y_2) \le \liminf_{n\to\infty}d_2(x_{2,n},y_2),
		\end{equation}
		we can subtract the $q$th power of \eqref{eqn:product-2} from the $q$th power of \eqref{eqn:product-1} to get:
		\begin{align*}
			d_1^q(x_1,y_1) &=  (d_1^q(x_1,y_1) + d_2^q(x_2,y_2)) -  d_2^q(x_2,y_2) \\ 
			&\ge  \lim_{n\to\infty}(d_1^q(x_{1,n},y_1) + d_2^q(x_{2,n},y_2)) -  \liminf_{n\to\infty}d_2^q(x_{2,n},y_2) \\
			&\ge  \limsup_{n\to\infty}((d_1^q(x_{1,n},y_1) + d_2^q(x_{2,n},y_2)) - d_2^q(x_{2,n},y_2)) \\
			&=  \limsup_{n\to\infty}d_1^q(x_{1,n},y_1).
		\end{align*}
		By combining this with (W2) for $w_1$, this implies $d_1(x_{1,n},y_1)\to d_1(x_1,y_1)$ as $n\to\infty$.
		Hence, (W3) for $w_1$ shows that we have $x_{1,n}\to x_1$ in $d_1$.
		By symmetry, the same argument shows that we have $x_{2,n}\to x_2$ in $d_2$.
		This shows $(x_{1,n},x_{2,n})\to (x_1,x_2)$ in $d_1\otimes_q d_2$, hence we have proven (W3) for $w_1\otimes w_2$.
		This complets the proof.
	\end{proof}
	
	Two extensions of this result are likely possible.
	First of all, since the metric $d_1\otimes_q d_2$ converges (say, pointwise) as $q\to\infty$, to
	\begin{equation*}
		(d_1\otimes_{\infty} d_2)((x_1,x_2),(x_1',x_2')):= \max\{d_1(x_1,x_1'),d_2(x_2,x_2')\},
	\end{equation*}
	it may be interesting to check whether $w_1\otimes w_2$ is also a weak convergence for $(X_1\times X_2,d_1\otimes_{\infty}d_2)$.
	Second, we note that one can likely prove a similar result for countable products of $(X_1,d_1)\otimes_q (X_2,d_2)\otimes_q \cdots$, simply by using a suitable diagaonliazation argument in the proof of (W1).
	However, we will not pursue these questions in the present work.
	
	Our next goal is to develop similar closure properties for quotients of metric spaces by certain group actions.
	It turns out that, in order to do this, we need to require a bit of compatibility between the action and the metric convergence and the action and the weak convergence.
	
	\begin{definition}
		For a set $X$, an action of a topological group $G$ on $X$ is said to be \textit{compatible with} a convergence $c$ on $X$ if
		\begin{equation*}
			x_n\to x \textnormal{ in } c \textnormal{ and } g_n\to g \textnormal{ implies } g_n\cdot x_n\to g\cdot x \textnormal{ in } c.
		\end{equation*}
		An action compatible with the metric convergence on a metric space will be called \textit{metrically compatible}, and an action compatible with a weak convergence for a metric space will be called \textit{weakly compatible}.
	\end{definition}
	
	Now we can state our result.
	For a metric space $(X,d)$ and a compact topological group $G$ acting isometrically on $(X,d)$, we define the metric
	\begin{equation*}
		d_{X/G}(u,u'):=\min_{x\in u,x'\in u'}d(x,x')
	\end{equation*}
	for $u,u'\in X/G$.
	Also, for a convergence $c$ on $X$, we write $c_{X/G}$ for the convergence on $X/G$ defined as follows:
	We say that $\{u_n\}_{n\in\N}$ and $u$ in $X/G$ have $u_n\to u$ in $c_{X/G}$ if for every $\{x_n\}_{n\in\N}$ in $X$ with $x_n\in u_n$ for all $n\in\N$, the sequence $\{x_n\}_{n\in\N}$ is relatively $w$-compact and any convergent subsequence thereof must lie in $u$;
	in other words, $c_{X/G}$ is exactly $\Gamma$-convergence of the orbits of $G$ under the convergence $c$.
	Then we get the following:
	
	\begin{proposition}\label{prop:quotient}
		Consider any metric space $(X,d)$ admitting a weak convergence $w$, and any compact topological group $G$ whose action on $(X,d)$ is isometric and weakly compatible.
		Then, $w_{X/G}$ is a weak convergence for $(X/G,d_{X/G})$.
	\end{proposition}
	
	\begin{proof}
		To show (W1), let us suppose that $\{u_n\}_{n\in\N}$ and $v$ in $X/G$ have
		\begin{equation*}
			\sup_{n\in\N}d_{X/G}(u_n,v) <\infty.
		\end{equation*}
		Now select an arbitrary $y\in v$ and get some $\{x_n\}_{n\in\N}$ in $X$ satisfying $x_n\in u_n$ and $d(x_n,y) = d(u_n,v)$ for all $n\in\N$.
		This means we have
		\begin{equation*}
			\sup_{n\in\N}d(x_n,y) <\infty,
		\end{equation*}
		so it follows from (W1) for $w$ that there exists $\{n_k\}_{k\in\N}$ and $x\in X$ such that we have $x_n\to x$ in $w$.
		Writing $u\in X/G$ for the orbit of $x$ under the action of $G$, we claim that we have $u_{n_k}\to u$ in $w_{X/G}$.
		To see this, let $\{x_k'\}_{k\in\N}$ be a sequence in $X$ with $x_k'\in u_{n_k}$ for all $k\in\N$.
		By definition, there exists some $g_k\in G$ for each $k\in\N$ such that we have $x_{k}' = g_k\cdot x_{n_k}$, and so we can use the fact that $G$ acts isometrically to compute:
		\begin{align*}
			d(x_k',y) &= d(g_k\cdot x_{n_k}, y) \\
			&= d(x_{n_k},g_k^{-1}\cdot y) \\
			&\le d(x_{n_k},y)+d(y,g_k^{-1}\cdot y) \\
			&= d(u_{n_k},v)+d(g_k\cdot y, y).
		\end{align*}
		In particular, this implies
		\begin{equation*}
			\sup_{k\in\N}d(x_k',y) \le \sup_{n\in\N}d(u_{n},v) + \sup_{g\in G}d(g\cdot y,y).
		\end{equation*}
		The first term is finite by assumption, and the second term is finite since $G$ is compact.
		Thus, we can apply (W1) for $w$ and the compactness of $G$ to get some $\{k_{j}\}_{j\in\N}$, $x'\in X$, and $g\in G$ such that we have $x_{k_j}'\to x'$ in $w$ and $g_{k_j}\to g$.
		It only remains to show that $x'\in u$, but this is easy:
		Since the action of $G$ is weakly compatible, we see that the sequence
		\begin{equation*}
			x_{k_j}' = g_{k_j}\cdot x_{n_{k_j}}
		\end{equation*}
		converges in $w$ to both $x'$ and $g\cdot x$, hence $x'=g\cdot x$.
		Since $x\in u$ by construction, this implies $x'\in u$.
		This proves (W1) for $w_{X/G}$.
		
		Next we show (W2) for $w_{X/G}$, so let us assume that $\{u_n\}_{n\in\N}$ and $u$ in $X/G$ have $u_n\to u$ in $w_{X/G}$, and that $v\in X/G$ is arbitrary.
		Then get some $\{n_k\}_{k\in\N}$ such that we have
		\begin{equation*}
			\liminf_{n\to\infty}d(u_n,v) = \lim_{k\to\infty}d(u_{n_k},v).
		\end{equation*}
		Also, let us fix $y\in v$ and get $\{x_k\}_{k\in\N}$ in $X$ such that we have $d(x_k,y) = d(u_{n_k},v)$ for all $k\in\N$.
		By the definition of $w_{X/G}$, this means there exists a further subsequence $\{k_j\}_{j\in\N}$ and some $x\in u$ such that we have $x_{k_j}\to x$ in $w$.
		Now applying (W2) for $w$, we get:
		\begin{align*}
			\liminf_{n\to\infty}d(u_n,v) &= \lim_{k\to\infty}d(u_{n_k},v) \\
			&= \lim_{k\to\infty}d(x_{k},y) \\
			&= \lim_{j\to\infty}d(x_{k_j},y) \\
			&\ge d(x,y) \\
			&\ge d_{X/G}(u,v).
		\end{align*}
		This shows (W2) for $w_{X/G}$, as needed.
		
		Finally, we show (W3) for $w_{X/G}$.
		So, suppose that $\{u_n\}_{n\in\N}$, $u$, and $v$ in $X/G$ have $u_n\to u$ in $w_{X/G}$ and
		\begin{equation*}
			d_{X/G}(u_n,v) \to d_{X/G}(u,v)
		\end{equation*}
		as $n\to\infty$.
		In order to show $d_{X/G}(u_n,u)\to 0$, it suffices to show that each subsequence of $\{u_n\}_{n\in\N}$ admits a further subsequence converging to $u$.
		So, we let $\{n_k\}_{k\in\N}$ be arbitrary.
		Then we fix $y\in v$, and let $\{x_k\}_{k\in\N}$ in $X$ satisfy $x_k\in u_{n_k}$ and $d(x_k,y) = d(u_{n_k},v)$ for all $k\in\N$.
		By the definition of $w_{X/G}$, there exists a subsequence $\{k_j\}_{j\in\N}$ and some $x\in u$ with $x_{k_j}\to x$ in $w$.
		By construction we have
		\begin{align*}
			\lim_{j\to\infty}d(x_{k_j},y) &= \lim_{k\to\infty}d(x_k,y) \\
			&= \lim_{n\to\infty}d(u_{n_k},v) \\
			&= \lim_{n\to\infty}d(u_n,v) \\
			&= d(u,v) \\
			&\le d(x,y),
		\end{align*}
		and by (W2) for $w$ we have
		\begin{align*}
			d(x,y)\le\liminf_{j\to\infty}d(x_{k_j},y).
		\end{align*}
		Combining these displays yields $d(x_{k_j},y)\to d(x,y)$, so (W3) for $w$ implies $x_{k_j}\to x$ in $d$.
		In particular, we have shown
		\begin{equation*}
			d_{X/G}(u_{n_{k_j}},u) \le d(x_{k_j},x) \to 0.
		\end{equation*}
		Since $\{n_k\}_{k\in\N}$ was arbitrary, we have shown
		\begin{equation*}
			d_{X/G}(u_{n},u) \to 0.
		\end{equation*}
		In other words, we have established (W3) for $w_{X/G}$.
		This finishes the proof.
	\end{proof}
	
	Our last result involves the notion of deformation of a metric space by a metric group, and this has statistical motivations that we will see in the next subsection.
	To state this precisely, let us say that $(G,\rho)$ is a \textit{metric group} if $\rho$ is a left-invariant metric inducing a topology which makes $G$ into a topological group, and let us say that $(G,\rho)$ is a \textit{proper metric group} if the metric space $(G,\rho)$ has the Heine-Borel property.
	If $(X,d)$ is a metric space and $(G,\rho)$ a proper metric group acting isometrically on $(X,d)$, then we define
	\begin{equation*}
		d_{G,\rho}(x,x') := \min_{g\in G}\sqrt{\rho^2(g,e) + d^2(x,g\cdot x')}
	\end{equation*}
	for $x,x'\in X$, where $e\in G$ is the identity element, which can easily be shown to be a metric on $X$.
	We refer to $(X,d_{G,\rho})$ as the \textit{deformation} of $(X,d)$ by $(G,\rho)$.
	
	Some remarks about this notion are due.
	First, we note that regularization is a sort of ``soft quotient'' whereas the standard quotient is a sort of ``hard quotient'', in the sense of ``soft'' and ``hard'' penalties in mathematical optimization.
	Second, we note that the term ``deformation'' is not standard, despite this operation being well-known in statistical shape analysis \cite{MillerYounes}; we use this term because this operation coincides with the notion of \textit{Cheeger deformation} if $(X,d)$ is a Riemannian manifold (see \cite[Chapter~6.1]{AlexandrinoBettiol}).
	Lastly, we mention that the existence of a suitable $\rho$ for a given group $G$ is, in great generality, guaranteed; it is known \cite{Struble1974} that every locally compact second countable Hausdorff group admits a left-invariant proper metric.
	This leads us to the last of our closure properties:
	
	\begin{proposition}\label{prop:partial-quotient}
		Consider any metric space $(X,d)$ admitting a weak convergence, and any proper metric group $(G,\rho)$ whose action on $(X,d)$ is isometric, metrically compatible, and weakly compatible.
		Then, $w$ is a weak convergence for $(X,d_{G,\rho})$.
	\end{proposition}
	
	\begin{proof}
		To show (W1), suppose that $\{x_n\}_{n\in\N}$ and $y$ in $X$ satisfy
		\begin{equation*}
			\sup_{n\in\N}d_{G,\rho}(x_n,y)<\infty.
		\end{equation*}
		Now note that for each $n\in\N$ there exists $g_n\in G$ such that $d_{G,\rho}^2(x_n,y) = \rho^2(g_n,e)+d^2(g_n\cdot x_n,y)$.
		The use the fact $G$ acts isometrically to derive the bound
		\begin{align*}
			d^2(x_n,y) &= d^2(g_n\cdot x_n,g_n\cdot y) \\
			&\le 2d^2(g_n\cdot x_n, y) + 2d^2(g_n\cdot y, y) \\
			&\le 2\rho^2(g_n,e)+2d^2(g_n\cdot x_n, y) + d^2(g_n\cdot y, y) \\
			&= 2d_{G,\rho}^2(x_n, y) + d^2(g_n\cdot y, y)
		\end{align*}
		for all $n\in\N$, hence
		\begin{equation*}
			\sup_{n\in\N}d^2(x_n,y) \le 2\sup_{n\in\N}d_{G,\rho}^2(x_n,y) + 2\sup_{n\in \N}d^2(g_n\cdot y, y).
		\end{equation*}
		Note that the first term is finite by assumption.
		Also, the second term is finite since $\sup_{n\in\N}\rho(g_n,e)<\infty$ and $(G,\rho)$ being proper implies that $\{g_n\}_{n\in\N}$ is compact.
		Thus, applying (W1) for $w$ on $(X,d)$ guarantees the existence of some $\{n_k\}_{k\in\N}$ and $x\in X$ with $x_n\to x$ in $w$.
		This proves (W1) for $w$ on $(X,d_{G,\rho})$.
		
		To show (W2), we suppose that $\{x_n\}_{n\in\N}$ and $x$ in $X$ have $x_n\to x$ in $w$ and we let $y\in X$ be arbitrary.
		Now let $\{n_k\}_{k\in\N}$ be a subsequence satisfying
		\begin{equation}\label{eqn:partial-quo-1}
			\liminf_{n\to\infty}d_{G,\rho}(x_n,y) = \lim_{k\to\infty}d_{G,\rho}(x_{n_k},y),
		\end{equation}
		and note that for each $k\in\N$ we can get some $g_k\in G$ such that $d_{G,\rho}^2(x_{n_k},y) = \rho^2(g_k,e) + d^2(g_k\cdot x_{n_k},y)$.
		Observe that, if the common value of \eqref{eqn:partial-quo-1} is infinite, then there is nothing to prove; so, we may assume that this value is finite, and this implies
		\begin{equation}\label{eqn:partial-quo-2}
			\limsup_{k\to\infty}\rho(g_k,e) \le \lim_{k\to\infty}d_{G,\rho}(x_{n_k},y) < \infty.
		\end{equation}
		Since $(G,\rho)$ is proper, equation~\eqref{eqn:partial-quo-2} implies that there exists a further subsequence $\{k_j\}_{j\in\N}$ and some $g\in G$ such that we have $g_{k_j}\to g$.
		By the weak compatibility of the group action, this implies
		\begin{equation*}
			g_{k_j}\cdot x_{n_{k_j}}\to g\cdot x \textnormal{ in } w.
		\end{equation*}
		Also, we of course have $\rho(g_{k_j},e)\to \rho(g,e)$.
		We can now put together these observations, along with (W2) for $w$ for $(X,d)$, to compute:
		\begin{align*}
			\liminf_{n\to\infty}d_{G,\rho}^2(x_n,y) &= \lim_{j\to\infty}d_{G,\rho}^2(x_{n_{k_j}},y) \\
			&= \lim_{j\to\infty}(\rho^2(g_{k_j},e)+d^2(g_{k_j}\cdot x_{n_{k_j}},y)) \\
			&= \lim_{j\to\infty}\rho^2(g_{k_j},e) +\liminf_{j\to\infty}d^2(g_{k_j}\cdot x_{n_{k_j}},y) \\
			&\ge \rho^2(g,e)+ d^2(g\cdot x,y) \\
			&\ge d_{G,\rho}^2(x,y).
		\end{align*}
		This proves (W2) for $w$ for $(X,d_{G,\rho})$, as claimed.
		
		Lastly, we prove (W3), so let us assume that $\{x_n\}_{n\in\N}$ and $x$ in $X$ have $x_n\to x$ in $w$ and that $y\in X$ satisfies $d_{G,\rho}(x_n,y)\to d_{G,\rho}(x,y)$.
		In order to show $d_{G,\rho}(x_n,x)\to 0$, it suffices to show that each subsequence of $\{x_n\}_{n\in\N}$ admits a further subsequece converging to $x$.
		So, let $\{n_k\}_{k\in\N}$ be arbitrary.
		Note that we have
		\begin{equation}\label{eqn:partial-quo-3}
			\sup_{k\in\N}d_{G,\rho}^2(x_{n_k},x) \le 2\sup_{k\in\N}d_{G,\rho}^2(x_{n_k},y) + 2d_{G,\rho}^2(x,y),
		\end{equation}
		and that the first term on the right is finite by the assumption that $d_{G,\rho}(x_n,y)\to d_{G,\rho}(x,y)$.
		Also, let us get, for each $k\in\N$ some $g_k\in G$ with $d_{G,\rho}^2(x_{n_k},x) = \rho^2(g_k,e) + d^2(g_k\cdot x_{n_k},x)$.
		Then \eqref{eqn:partial-quo-3} implies
		\begin{equation*}
			\sup_{k\in\N}(\rho^2(g_k,e) + d^2(g_k\cdot x_{n_k},x)) < \infty.
		\end{equation*}
		Now we can apply (W3) for $w$ for $(X,d)$ and the properness of $(G,\rho)$ to get a subsequence $\{k_j\}_{j\in\N}$ and some $g\in G$ and $x\in X$ satisfying $g_{k_j}\to g$ and
		\begin{equation*}
			x_{n_{k_j}}\to x \textnormal{ in } w.
		\end{equation*}
		By the weak compatibility of the action of $G$, this further implies
		\begin{equation}\label{eqn:partial-quo-4}
			g_{k_j}\cdot x_{n_{k_j}}\to g\cdot x \textnormal{ in } w.
		\end{equation}
		Now we compute:
		\begin{align*}
			\rho^2(g,e) + d^2(g\cdot x, y) &\le \lim_{j\to\infty}\rho^2(g_{k_j},e) + \liminf_{j\to\infty}d^2(g_{k_j}\cdot x_{n_{k_j}},y) \\
			&= \liminf_{j\to\infty}\left(\rho^2(g_{k_j},e) +d^2(g_{k_j}\cdot x_{n_{k_j}},y)\right) \\
			&= \liminf_{j\to\infty}d_{G,\rho}^2(x_{n_{k_j}},y) \\
			&= \lim_{n\to\infty}d_{G,\rho}^2(x_n,y) \\
			&= d_{G,\rho}^2(x,y).
		\end{align*}
		Since $d_{G,\rho}(x,y)$ is defined via a minimum, this in fact implies
		\begin{equation*}
			\rho^2(g,e) + d^2(g\cdot x, y) = d_{G,\rho}^2(x,y).
		\end{equation*}
		Consequently, our assumption of $d_{G,\rho}(x_n,y)\to d_{G,\rho}(x,y)$ yields
		\begin{equation}\label{eqn:partial-quo-5}
			\rho^2(g_{k_j},e) + d^2(g_{k_j}\cdot x_{n_{k_j}},y) \to \rho^2(g,e) + d^2(g\cdot x,y)
		\end{equation}
		as $j\to\infty$.
		But we of course have $\rho^2(g_{k_j},e) \to \rho^2(g,e)$ by construction, so subtracting this from \eqref{eqn:partial-quo-5} gives
		\begin{equation}\label{eqn:partial-quo-6}
			d(g_{k_j}\cdot x_{n_{k_j}},y) \to d(g\cdot x,y).
		\end{equation}
		However, by combining \eqref{eqn:partial-quo-4} and \eqref{eqn:partial-quo-6} with (W3) for $w$ on $(X,d)$, we conclude 
		\begin{equation*}
			d(g_{k_j}\cdot x_{n_{k_j}},g\cdot x)\to 0
		\end{equation*}
		as $j\to\infty$.
		But also we have $g_{k_j}^{-1}\to g^{-1}$ so the metric compatibility of the action $G$ yields
		\begin{equation*}
			x_{n_{k_j}} = g_{k_j}^{-1}\cdot g_{k_j}\cdot x_{n_{k_j}} \to g^{-1}\cdot g\cdot x = x \textnormal{ in } d
		\end{equation*}
		as $j\to\infty$.
		In other words, we have shown
		\begin{equation*}
			d(x_{n_{k_j}},x)\to 0
		\end{equation*}
		as $j\to\infty$.
		Since $\{n_k\}_{k\in\N}$ was arbitrary, this implies $d(x_n,x)\to 0$ as $n\to\infty$.
		This shows (W3) for $w$ for $(X,d_{G,\rho})$, and finishes the proof.
	\end{proof}
	
	In the setting of Proposition~\ref{prop:partial-quotient} one may of course replace the proper metric $\rho$ with $\lambda^{-1} \rho$ for any $0<\lambda<\infty$, leading to a deformed metric of the form:
	\begin{equation*}
		d_{G,\lambda^{-1}\rho}(x,x') := \min_{g\in G}\sqrt{\frac{1}{\lambda^2}\rho^2(g,e) + d^2(x,g\cdot x')}
	\end{equation*}
	for $x,x'\in X$.
	The extremal cases of this parameterization are interesting to consider:
	As $\lambda \to 0$ we have $d_{G,\lambda^{-1}\rho}$ converging pointwise to the given metric $d$.
	As $\lambda\to \infty$ we have $d_{G,\lambda^{-1}\rho}$ converging pointwise to a pseudometric on $X$ which descends to the metric $d_{X/G}$ on the quotient space $X/G$.
	(In fact, if $(X,d)$ is a compact Riemannian manifold, then \cite[p. 143]{AlexandrinoBettiol} $d_{G,\lambda^{-1}\rho}$ converges to $d_{X/G}$ in the Gromov-Hausdorff sense as $\lambda\to \infty$.)
	As such, it is interesting to observe that the family of metrics $\{d_{G,\lambda^{-1}\rho}\}_{\lambda>0}$ on $X$, which in some sense interpolate between the metrics $d$ and $d_{X/G}$, share a single notion of weak convergence.
	
	\subsection{Further Examples}\label{subsec:further}
	
	By combining the results of Subsection~\ref{subsec:ex} and Subsection~\ref{subsec:closure}, we have a calculus for constructing a wide range of metric spaces admitting a weak convergence.
	In this final subsection we give a few examples of this; they are primarily motivated by statistical concerns that we will see later in the paper.
	
	First is a metric where Fr\'echet means have been applied in problems in computational tomography \cite{WassersteinTomography}, which has also been studied in recent works \cite{WassersteinAlignment,ProcrustesWasserstein}.
	
	\begin{example}[rotationally-invariant Wasserstein space]
		Let $(\Pcal_2(\R^m),W_2)$ denote the Wasserstein space of probability measures on $\R^m$ with finite variance, and let $\ortho(m)$ denote the space of orientation-preserving $m\times m$ rotation matrices
		Then let $\ortho(m)$ act on $\Pcal_2(\R^m)$ via $(U\cdot P)(A) := P(U^{-1}A)$, and define the pseudometric
		\begin{equation*}
			d(P,Q) := \min_{U\in \ortho(m)}W_2(U\cdot P,Q)
		\end{equation*}
		on $\Pcal_2(\R^m)$, which descends to a bona fide metric on the resulting quotient $\Pcal$ of $\Pcal_2(\R^m)$.
		From Example~\ref{ex:wass} we know that $(\Pcal_2(\R^m),W_2)$ admits a weak convergence, and it is straightforward to check that this action is isometric and weakly compatible; thus, Proposition~\ref{prop:quotient} implies that $(\Pcal,d)$ admits a weak convergence.
	\end{example}
	
	Second is a metric space arising in image classification problems originally studied in pattern theory \cite[Section~5.3]{MillerYounes}.
	
	\begin{example}[deformation of images by rotations]
		For $m\in\N$ and $1<q<\infty$, consider the space $L^q(\R^m;\R)$, regarded as a space of images where, for $f\in L^q(\R^m;\R)$ and $x\in \R^m$, we think of $f(x)$ as the greyscale intensity of $f$ at $x$.
		(Usually we have $m=2,3$ in applications.)
		Now define the metric $d$ on $L^q(\R^m;\R)$ via
		\begin{equation*}
			d(f,g) := \min_{A\in \skewsym(m)}\left(\|A\|_{\textnormal{F}}^2 + \left(\int_{\R^m}|f(e^{-A}x) - g(x)|^q\diff x\right)^{2/q}\right)^{1/2},
		\end{equation*}
		where $\skewsym(m)$ denotes the vector space of skew-symmetric $m\times m$ matrices.
		
		We notice that the metric $d$ can be realized as the deformation of $L^q(\R^m;\R)$ by the precomposition action of the group of orientation-preserving rotations $\ortho(m)$, where the metric $\rho$ is taken to be $\rho(U,I) := \|\log U\|$ for $U\in\ortho(m)$; since $\log$ maps $\ortho(m)$ to $\skewsym(m)$, we can parameterize the group via the latter.
		Also, it is straightforward to see that this group action is both weakly and metrically compatible.
		Thus, we can deduce that $(L^q(\R^m;\R),d)$ admits a weak convergence simply by combining Proposition~\ref{prop:partial-quotient} with Example~\ref{ex:unif-cvx-Banach}.
	\end{example}
	
	Thid is particular metric on the space of persistence diagrams  \cite{DivolLacombe}, which has become a central element of topological data analysis.
	
	\begin{example}[partial matching metrics on persistence diagrams]
		Define the set $D:=\{(x_1,x_2)\in\R^2: x_1<x_2\}$ which represents the open super-diagonal in $\R^2$, and for $1<q<\infty$ write $\mathcal{D}_q$ for the space of countable, locally-finite multisets $P$ of $D$ satisfying $\sum_{y\in P}\min_{x\in \Delta}\|x-y\|^q<\infty$, where $\Delta:=\{(x_1,x_2)\in\R^2: x_1=x_2\}$ is the diagonal in $\R^2$.
		For $P,P'\in\mathcal{D}_q$, we define
		\begin{equation*}
			B_q(P,P'):= \min_{\phi}\left(\sum_{x\in P\cup\Delta}\|x-\phi(x)\|^q\right)^{1/q},
		\end{equation*}
		where $\phi$ is taken over all bijections between $P\cup\Delta$ and $P'\cup\Delta$.
		We regard $\mathcal{D}_q$ as the set of persistence diagrams with finite $q$th moment, and we note that $B_q$ is a metric on the space $\mathcal{D}_q$
		
		We claim that $(\mathcal{D}_q,B_q)$ admits a weak convergence the sense of Definition~\ref{def:weak-like}.
		To see this, let us consider the space $\mathcal{M}_q$ of all Radon measures on $D$, and the metric $O_q$ of partial optimal transport \cite{DivolLacombe} with respect to $\Delta$.
		We claim that the topology of vague convergence is a weak convergence for $(\mathcal{M}_q,O_q)$.
		To see (W1), simply note that, if $\{P_n\}_{n\in\N}$ and $Q$ in $\mathcal{M}_q$ satisfy
		\begin{equation*}
			\sup_{n\in\N}O_q(P_n,Q) < \infty,
		\end{equation*}
		then we can apply the triangle inequality for $O_q$ and consider any Radon measure supported on $\Delta$ to conclude
		\begin{equation*}
			\sup_{n\in\N}\int_{D}\min_{x\in\Delta}\|x-y\|^q\diff P_n(y)<\infty.
		\end{equation*}
		In particular, combining this with \cite[Remark~3.2]{DivolLacombe} implies $\sup_{n\in\N}|P_n|(K) < \infty$ for all compact $K\subseteq D$, so Helly's selection theorem guarantees the existence of a vaguely convergent subsequence of $\{P_n\}_{n\in\N}$.
		Also note that (W2) is proven in \cite[Proposition~3.1]{DivolLacombe} and that (W3) is proven in \cite[Theorem~3.4]{DivolLacombe}.
		
		Now note that we can identify $\mathcal{D}_q$ with a subset of $\mathcal{M}_q$ by identifying each countable, locally-finite multiset $P$ with a counting measure supported on $P$.
		Moreover, it is known \cite[Proposition~A.5]{DivolLacombe} that, under this identification, $\mathcal{D}_q$ is a vaguely closed subset of $\mathcal{M}_q$.
		Therefore, the result follows from Proposition~\ref{prop:restriction}.
	\end{example}
	
	Next is a metric on a suitable space of smooth loops in the plane; we think of these as the boundaries of two-dimensional images, and they have been extensively studied in applied and computational aspects of statistical shape analysis \cite{SrivastavaKlassen}.
	
	\begin{example}[Sobolev metric on unparameterized unregistered planar loops]\label{ex:loops}
		For $\mathbb{S}^1=[0,2\pi)$ endowed with arithmetic modulo $2\pi$, consider the collection $\tilde{\mathcal{L}}$ of all functions $f:\mathbb{S}^{1}\to\R^2$ which are twice differentiable almost everywhere, have first derivative satisfying $\|f'(t)\| =1$ for all $0\le t \le 2\pi$, and which satisfy the condition
		\begin{equation*}
			\int_{0}^{2\pi}\left(\|f(t)\|^2+\|f''(t)\|^2\right)\diff t< \infty.
		\end{equation*}
		On the space $\tilde{\mathcal{L}}$ we define the pseudo-metric
		\begin{equation*}
			\tilde{d}(f,g):=\min_{\substack{\theta\in\mathbb{S}^1 \\ U\in\ortho(2)}}\left(\int_{0}^{2\pi}\sum_{k=0}^{2}\|Uf^{(k)}(t-\theta) -g^{(k)}(t)\|^2\diff t\right)^{1/2},
		\end{equation*}
		which descends to a bona fide metric $d$ on the natural quotient $\mathcal{L}$ of $\tilde{\mathcal{L}}$.
		
		We claim that the metric space $(\mathcal{L},d)$ admits a weak convergence, and we prove this by building it step-by-step using the results of the previous subsections.
		First, note that the Sobolev space $H^{2,2}(\mathbb{S}^{1};\R^2)$ is a Hilbert space hence, by Example~\ref{ex:unif-cvx-Banach} or Example~\ref{ex:hadamard}, its usual weak convergence is a weak convergence in the sense of Definition~\ref{def:weak-like}.
		Next we observe (say, by integrating against a suitable test function and using integration by parts) that the map $e_t:H^{2,2}(\mathbb{S}^{1};\R^2)\to \R$ defined via $e_t(f):= f'(t)$ is weakly continuous.
		Thus, the space $\tilde{\mathcal{L}}$ is a weakly closed subset of $H^{2,2}(\mathbb{S}^{1};\R^2)\to \R$, so it follows from Proposition~\ref{prop:restriction} that $\tilde{\mathcal{L}}$, when endowed with the metric from the ambient space $H^{2,2}(\mathbb{S}^{1};\R^2)$, admits a weak topology via the restriction of the weak convergence of $H^{2,2}(\mathbb{S}^{1};\R^2)$.
		Lastly, we let the compact group $\mathbb{S}^1\times \ortho(2)$ act isometrically on $\tilde{\mathcal{L}}$ via
		\begin{equation*}
			((\theta,U)\cdot f)(t):= Uf(t-\theta),
		\end{equation*}
		and we observe that $d$ is exactly the quotient metric
		\begin{equation*}
			d_{\tilde{\mathcal{L}}/(\mathbb{S}^1\times \ortho(2))}.
		\end{equation*}
		In fact, one can easily see that the action of $\mathbb{S}^1\times \ortho(2)$ on $\tilde{\mathcal{L}}$ is compatible with the weak convergence.
		Thus, Proposition~\ref{prop:quotient} shows that $(\mathcal{L},d)$ admits a weak convergence.
	\end{example}
	
	We note that, in the preceding example, there are many different ways to formulate similar metrics on spaces of curves; see \cite{MichorMumford, Kurtek, Sundaramoorthi, BruverisMichorMumford, TumpachPreston}.
	We believe an interesting avenue for future work would be to verify the existence of a weak convergence in these other models.
	
	Last is a metric on the space of covariance operators on a Hilbert space, which has recently been studied in functional analysis, optimal transport, and quantum information theory \cite{PanaretosSantoroBW, Kroshnin, Masarotto, ZemelPanaretosProcrustes}.
	
	\begin{example}[Bures-Wasserstein metric on infinite-dimensional Hilbert space]\label{ex:BW-inf-dim}
		Let $\hilbert$ denote a real, separable, infinite-dimensional Hilbert space, with norm $\|\cdot\|$ and inner product $\langle\,\cdot\,,\,\cdot\,\rangle$.
		We write $\covspace$ for the space of covariance operators on $\hilbert$, that is, the space of all non-negative, self-adjoint, trace-class linear operators from $\hilbert$ to itself.
		The \textit{Bures-Wasserstein metric} $\Pi$ on $\covspace$ is defined via
		\begin{equation*}
			\Pi(\Sigma,\Sigma') :=\sqrt{ \trace(\Sigma)+\trace(\Sigma')-2\trace({\Sigma}^{\frac{1}{2}}\Sigma'{\Sigma}^{\frac{1}{2}})^{\frac{1}{2}}},
		\end{equation*}
		although it is known to have several equivalent formulations; see \cite{Masarotto}.
		Importantly, note that, if we identify $\Sigma,\Sigma'\in\covspace$ with the centered Gaussian measures $\mu,\mu'$ on $\hilbert$ with covariance operators given by $\Sigma,\Sigma'$, respectively, then we have $\Pi(\Sigma,\Sigma')= W_2(\mu,\mu')$.
		
		We claim that the metric space $(\covspace,\Pi)$ admits a weak convergence.
		To show this, we first show the that $(\Pcal_2(\hilbert),W_2)$ admits a weak convergence.
		To define this, let $\{e_i\}_{i\in\N}$ denote an orthonormal basis for $(\hilbert,\|\cdot\|)$, and define the pre-Hilbert norm
		\begin{equation*}
			\|x\|' := \sqrt{\sum_{i=1}^{\infty}i^{-2}\langle x,e_i\rangle^2}
		\end{equation*}
		for $x\in \hilbert$; we say that $\{P_n\}_{n\in\N}$ and $P$ in $\Pcal_2(\hilbert)$ satisfy $P_n\to P$ in $w$ if we have $\int_{\hilbert}\phi\diff P_{n} \to \int_{\hilbert}\phi\diff P$ for all bounded continuous functions $\phi:(\hilbert,\|\cdot\|')\to\R$.
		We claim that $w$ is indeed a weak convergence for $(\Pcal_2(\hilbert),W_2)$ in the sense of Definition~\ref{def:weak-like}.
		(We emphasize that $W_2$ always refers to the Wasserstein space with respect to the norm $\|\cdot\|$, although this does not appear in the notation.)
		To see (W1), note that if $\{P_n\}_{n\in\N}$ and $Q$ in $\Pcal_2(\hilbert)$ satisfy
		\begin{equation*}
			\sup_{n\in\N}W_2(P_n,Q)<\infty,
		\end{equation*}
		then the same argument as in Example~\ref{ex:wass} shows
		\begin{equation*}
			\sup_{n\in\N}P_n(\hilbert\setminus B_{\ell}(0))\to 0
		\end{equation*}
		as $\ell\to\infty$, where $B_{\ell}(0):=\{x\in \hilbert: \|x\|\le \ell\}$ is the closed ball of radius $\ell$ in $\hilbert$ with respect to $\|\cdot\|$; consequently, (W1) follows from \cite[Lemma~5.1.12(c)]{AmbrosioGigliSavare}.
		Moreover, (W2) is proven in \cite[Lemma~7.1.4]{AmbrosioGigliSavare}.
		Finally, we note that (W3) can be proven with the exact same proof as \cite[Lemma~14]{WassersteinBarycenter}, simply by noting that both the gluing lemma and the Skorokhod representation theorem hold for tight measures on a separable (possibly incomplete) metric space.
		In order to complete the proof, it suffices by Proposition~\ref{prop:restriction} to show that the Gaussian measures are $w$-closed in $\Pcal_2(\hilbert)$.
		To do this, suppose that $\{\mu_n\}_{n\in\N}$ and $\mu\in\Pcal_2(\hilbert)$ satisfy $\mu_{n}\to \mu$ in $w$, and that there exist $\{\Sigma_n\}_{n\in\N}$ in $\covspace$ such that $\mu_n$ can be identified with the centered Gaussian measure with covariance operator $\Sigma_n$ for each $n\in\N$.
		Now apply \cite[Lemma~5.1.12(f)]{AmbrosioGigliSavare} to see that all finite-dimensional distributions of $\mu$ are Gaussian, hence $\mu$ is a Gaussian measure.
		This completes the proof.
	\end{example}
	
	\section{The Main Theorem}\label{sec:main}
	
	This section contains the main theorems of the paper.
	That is, we prove suitable convergence of the Fr\'echet mean sets under suitable convergence of the underlying measures (Theorem~\ref{thm:main}); as a consequence, we show compactness of the Fr\'echet mean set (Corollary~\ref{cor:main-cpt}) and uniform consistency of the Fr\'echet mean sets under suitable convergence of the underlying measures (Corollary~\ref{cor:main-cty}).
	We emphasize that the results in this section are purely analytic since they require only convergence of the underlying measures; in the next section we will consider the probabilistic setting of empirical measures of some random variables.
	
	Throughout this section, $(X,d)$ denotes a metric space.
	Unless otherwise stated, all metric and topological notions refer to the topology generated by the metric $d$.
	We write $B_r^{\circ}(x) := \{y\in X: d(x,y)<r\}$ and $\bar B_r(x) := \{y\in X: d(x,y)<r\}$ for the open and closed balls of radius $r>0$ around $x\in X$.
	
	Now let us introduce some notation in order to state our results.
	For a metric space $(X,d)$, we write $\Pcal(X)$ for the set of Borel probability measures on $X$, with respect to the topology generated by the metric $d$.
	For $r\ge 0$, we write $\Pcal_r(X)\subseteq \Pcal(X)$ for the set of $\mu\in \Pcal(X)$ satisfying $\int_{X}d^r(x,y)\diff \mu(y)<\infty$ for some $x\in X$.
	(This holds for some $x\in X$ if and only if it holds for all $x\in X$; see \eqref{eqn:ineq-1}.)
	We regard $\Pcal_r(X)$ as the set of distributions on $X$ with $r$ moments finite.
	
	For a metric space $(X,d)$, we write $\weak$ for the topology on $\Pcal(X)$ such that $\{\mu_n\}_{n\in\N}$ and $\mu$ in $\Pcal(X)$ have $\mu_n\to \mu$ in $\weak$ if and only if we have $\int_{X}f\diff\mu_n\to\int_{X}f\diff\mu$ for all bounded continuous $f:(X,d)\to\R$.
	Further, for $r\ge 0$, we write $\weak^r$ for the topology on $\Pcal_r(X)$ such that $\{\mu_n\}_{n\in\N}$ and $\mu$ in $\Pcal_r(X)$ have $\mu_n\to \mu$ in $\weak^r$ if and only if we have $\mu_n\to \mu$ in $\weak$ and $\int_{X}d^r(x,y)\diff\mu_n(y)\to\int_{X}d^r(x,y)\diff\mu(y)$ for all $x\in X$; it turns out \cite[Lemma~2.1]{EvansJaffeFrechet} that, if $\mu_n\to \mu$ in $\weak$, then $\int_{X}d^{r}(x,y)\diff\mu_n(x)\to \int_{X}d^{r}(x,y)\diff\mu(x)$ for some $x\in X$ is equivalent to $\int_{X}d^{r}(x,y)\diff\mu_n(x)\to \int_{X}d^{r}(x,y)\diff\mu(x)$ for all $x\in X$.
	We also write $\Pcal_0(X):=\Pcal(X)$ and $\weak^0 :=\weak$.
	
	Now we introduce two fundamental inequalities.
	First, define the constant $c_r := \max\{1,2^{r-1}\}$ for $r\ge 0$ which has the property that for all $x,x',x''\in X$ we have
	\begin{equation}\label{eqn:ineq-1}
		d^r(x,x'') + c_r(d^r(x,x') + d^r(x',x'')).
	\end{equation}
	Second, for all $p\ge 1$ and all $x,x',x''\in X$ we have
	\begin{equation}\label{eqn:renorm}
		|d^p(x,x'')-d^p(x',x'')| \le pd(x,x')(d^{p-1}(x,x'')+d^{p-1}(x',x'')).
	\end{equation}
	These results are elementary, and proofs can be found in \cite[Equation~(8)]{EvansJaffeFrechet} and \cite[Lemma~2.6]{EvansJaffeFrechet}, respectively.
	
	Next we introduce the central notion of the paper.
	
	\begin{definition}
		For any metric space $(X,d)$ and any $p\ge1$, we define the function $W_p:\Pcal_{p-1}(X)\times X^2\to\R$ for $p \ge 1$ via
		\begin{equation*}
			W_p(\mu,x,x'):=\int_{X}(d^p(x,y)-d^p(x',y))\diff\mu(y),
		\end{equation*}
		called the \textit{Fr\'echet $p$-functional of $\mu$.}
		Observe that the value of the integral is well-defined, because of \eqref{eqn:renorm}.
	\end{definition}
	
	\begin{definition}
		For any metric space $(X,d)$, any $p\ge 1$, and any $\mu\in\Pcal_{p-1}(X)$, we write $\obj_p(\mu)\subseteq X$ for the set
		\begin{equation*}
			\obj_{p}(\mu) := \left\{x\in X : W_p(\mu,x,x')\le 0 \textrm{ for all } x'\in X \right\},
		\end{equation*}
		called the \textit{Fr\'echet $p$-mean set of $\mu$.}
		Observe that this set can be empty, it can be a singleton, or it can consist of more than one point.
	\end{definition}
	
	In case this notion is somewhat opaque, we include the following remarks:
	First, if it happens that $\mu\in\Pcal_p(X)$ instead of just $\mu\in\Pcal_{p-1}(X)$, then we in fact have
	\begin{equation*}
		\obj_p(\mu) = \underset{x\in X}{\arg\min}\int_{X}d^p(x,y)\diff\mu(y).
	\end{equation*}
	For example, this holds whenever $\mu$ takes the form of an empirical measure $\frac{1}{n}\sum_{i=1}^{n}\delta_{y_i}$ for some points $y_1,\ldots, y_n\in X$, since then $\mu\in\bigcap_{p\ge 1}\Pcal_p(X)$.
	Second, in the general case of $\mu\in\Pcal_{p-1}(X)$, we in fact have
	\begin{equation*}
		\obj_p(\mu) = \underset{x\in X}{\arg\min}\int_{X}(d^p(x,y)-d^p(o,y))\diff\mu(y) = \underset{x\in X}{\arg\min}\,W_p(\mu,x,o)
	\end{equation*}
	for any $o\in X$, which we regard as a choice of ``origin'', although the set of minimizers does not depend on $o$.
	
	Lastly, we will need the following:
	
	\begin{definition}
		For any metric space $(X,d)$, any $p\ge 1$, any $\mu\in\Pcal_{p-1}(X)$, and any $o\in X$ we set
		\begin{equation*}
			\val_{p}(\mu,o) := \inf_{x\in X}W_p(\mu,x,o),
		\end{equation*}
		called the \textit{Fr\'echet $p$-variance of $\mu$ with origin $o$.}
	\end{definition}
	
	Of course the connection between these notions is that we have
	\begin{equation*}
		\obj_p(\mu) = \left\{x\in X: W_p(\mu,x,o) = \val_p(\mu,o) \right\}
	\end{equation*}
	for all $o\in X$.
	However, note that the level $\val_p(\mu,o)$ depends on the choice of $o\in X$ even though the level set $\obj_p(\mu)$ does not.
	
	Next we state a small technical results that we will need later.
	Their proofs are standard in the case that $(X,d)$ is Polish, but we note that they easily extend to the general separable setting, and we omit the proof for the sake of brevity.
	
	\begin{lemma}\label{lem:technical-2}
		If $(X,d)$ is a separable metric space and $0\le r < \infty$, and if $\{\mu_n\}_{n\in\N}$ and $\mu$ in $\Pcal_r(X)$ have $\mu_n\to \mu$ in $\weak^r$, then we have
		\begin{equation*}
			\lim_{L\to\infty}\limsup_{n\to\infty}\int_{X\setminus B^{\circ}_{L}(o)}d^{r}(o,y)\diff\mu_{n}(y) = 0
		\end{equation*}
		for all $o\in X$.
	\end{lemma}
	
	Now we can prove a fundamental result which allows us to conclude that certain Fr\'echet mean sets and their ``relaxations'' (see \cite{FrechetRelaxation}) are uniformly bounded.
	To state this, let us write
	\begin{equation*}
		\obj_p(\mu;\varepsilon) = \left\{x\in X: W_p(\mu,x,o) \le \val_p(\mu,o) +\varepsilon\right\}
	\end{equation*}
	for all $p\ge1,\mu\in\Pcal_{p-1}(X),o\in X$, and $\varepsilon\ge 0$; note that we of course have $\obj_p(\mu;0) = \obj_p(\mu)$, and also that $\obj_p(\mu;\varepsilon)$ is non-empty, by definition, for $\varepsilon>0$.
	We can also equivalently write
	\begin{equation*}
		\obj_{p}(\mu;\varepsilon) = \left\{x\in X : W_p(\mu,x,x')\le \varepsilon \textrm{ for all } x'\in X \right\},
	\end{equation*}
	which we previously saw in the case of $\varepsilon = 0$.
	Then we have the following, which is a slight adaptation of \cite[Lemma~C.1]{FrechetRelaxation}.
	
	\begin{lemma}\label{lem:obj-bdd}
		Consider any separable metric space $(X,d)$ and any $p\ge 1$.
		If $\{\mu_n\}_{n\in\N}$ and $\mu$ in $\Pcal_{p-1}(X)$ have $\mu_n\to \mu$ in $\weak^{p-1}$ and if $\{\varepsilon_n\}_{n\in\N}$ is any bounded sequence of non-negative real numbers, then there exists a $d$-bounded set $B\subseteq X$ satisfying $M_p(\mu_n;\varepsilon_n)\subseteq B$ for all $n\in\N$.
	\end{lemma}
	
	\begin{proof}
		Fix an arbitrary $o\in X$.
		First, we use Lemma~\ref{lem:technical-2} for $r=p-1$ and $r=0$ to choose $L>0$ large enough so that we have
		\begin{equation}\label{eqn:moment-bdd-4}
			\int_{X\setminus B^{\circ}_{L}(o)}d^{p-1}(o,y)d\mu_{n}(y) \le \frac{1}{p2^{2p+2}}
		\end{equation}
		and
		\begin{equation}\label{eqn:moment-bdd-3}
			\mu_{n}(X\setminus B^{\circ}_{L}(o)) \le \frac{1}{p2^{2p}} \le \frac{1}{2}.
		\end{equation}
		for all $n\in\N$.
		Second, we derive some pointwise inequalities.
		For one, use the elementary inequality \eqref{eqn:ineq-1} to get
		\begin{equation}\label{eqn:moment-bdd-2}
			\begin{split}
				d^p(x,y)-d^p(o,y) &\ge c_p^{-1}d^p(x,o) - 2d^p(o,y) \\
				&\ge \frac{d^p(x,o)}{2^{p-1}} - 2d^p(o,y) \\
			\end{split}
		\end{equation}
		for all $x,y\in X$ and $n\in\N$.
		For another, use \eqref{eqn:renorm} then \eqref{eqn:ineq-1} to get
		\begin{equation*}
			\begin{split}
				&|d^p(x,y)-d^p(o,y)| \\
				& \le pd(x,o)(d^{p-1}(x,y)+d^{p-1}(o,y)) \\
				& \le pd(x,o)(c_{p-1}(d^{p-1}(x,o)+d^{p-1}(o,y))+d^{p-1}(o,y)) \\
				&= pc_{p-1}d^p(x,o)+p(1+c_{p-1})d(x,o)d^{p-1}(o,y) \\
				&\leq p2^{p-1}d^p(x,o)+p2^pd(x,o)d^{p-1}(o,y)
			\end{split}
		\end{equation*}
		hence
		\begin{equation}\label{eqn:moment-bdd-1}
			d^p(x,y)-d^p(o,y) \ge -p2^{p-1}d^p(x,o)-p2^pd(x,o)d^{p-1}(o,y)
		\end{equation}
		for all $x,y\in X$.
		
		Now we claim that we have $x_n\in B_{r}(o)$ for all $n\in\N$, where
		\begin{equation*}
			r:=2\left(2^{p+1}L^p + \sup_{n\in\N}\varepsilon_n+1\right).
		\end{equation*}
		To see this, assume for the sake of contradiction that some $N\in\N$ has $d(x_N,o) > r$.
		Then we can use \eqref{eqn:moment-bdd-2} and \eqref{eqn:moment-bdd-3} to get
		\begin{align*}
			\int_{B^{\circ}_{L}(o)}&(d^p(x_{N},y)-d^p(o,y))\, d\mu_{N}(y) \\
			&\geq \left(\frac{d^p(x_{N},o)}{2^{p-1}} - 2L^p \right)\mu_{N}(B_{L}^{\circ}(o)) \\
			&\geq \left(\frac{d^p(x_{N},o)}{2^{p-1}} - 2L^p \right)\frac{1}{2} \\
			&=\frac{d^p(x_{N},o)}{2^{p}} - L^p,
		\end{align*}
		where in the first inequality we used $d(x_{N},o) \ge 2^{p+1}L^p$ to see that the integrand is non-negative.
		Also, we can use \eqref{eqn:moment-bdd-1} and \eqref{eqn:moment-bdd-4} to compute
		\begin{align*}
			\int_{X\setminus B^{\circ}_{L}(o)}&(d^p(x_{N},y)-d^p(o,y))\, d\mu_{N}(y) \\
			&\ge -p 2^{p-1} d^p(x_{N},o)\mu_N(X\setminus B_{L}^{\circ}(o)) \\
			&\qquad - p2^p d(x_{N},o) \int_{X\setminus B_{L}(o)}d^{p-1}(o,y)\,d\mu_{N}(y) \\
			&\geq  -\frac{d^p(x_{N},o)}{2^{p+1}} - \frac{d(x_{N},o)}{2^{p+2}}.
		\end{align*}
		By combining the previous two displays, we get
		\begin{align*}
			W_p(\mu_{N},x_{N},o) &= \int_{B^{\circ}_L(o)}(d^p(x_{N},y)-d^p(o,y))\, d\mu_{N}(y) \\
			&\qquad+ \int_{X\setminus B^{\circ}_{L}(o)}(d^p(x_{N},y)-d^p(o,y))\, d\mu_{N}(y) \\
			&\ge\frac{d^p(x_{N},o)}{2^{p}} - L^p   -\frac{d^p(x_{N},o)}{2^{p+1}} - \frac{d(x_{N},o)}{2^{p+2}} \\
			&=\frac{1}{2^{p+1}}\left(d^p(x_{N},o) - \frac{1}{2}d(x_{N},o) - 2^{p+1}L^p \right).
		\end{align*}
		Thus, this implies
		\begin{equation}\label{eqn:bdd-nonneg}
			d^p(x_{N},o) - \frac{1}{2}d(x_{N},o) - 2^{p+1}L^p \le \varepsilon_N \le \sup_{n\in\N}\varepsilon_n
		\end{equation}
		since we have $x_N\in M_p(\mu_N;\varepsilon_N)$ be assumption.
		
		Finally, we show that this gives a contradiction.
		To do this, we claim that $u\ge 0$ satisfying $u^p-\frac{1}{2}u-c\le 0$ implies $u\le 2(c+1)$.
		Indeed, the function $f(u):= u^p-\frac{1}{2}u-c$ is convex and hence is lower bounded by each of its linear approximations $\ell$; by taking the linear approximation at the point $u_p:=p^{-1/(p-1)}$ and using $\{u\ge 0: f(u)\le 0\}\subseteq \{u\ge 0: \ell(u) \le 0\}$, we see that $f(u)\le 0$ implies
		\begin{equation*}
			u \le  u_p-2f(u_p) \le  p^{-\frac{1}{p-1}}-2\left(p^{-\frac{p}{p-1}} - \frac{1}{2}p^{-\frac{1}{p-1}}+2c\right) \le 2\left(c+p^{-\frac{1}{p-1}}\right).
		\end{equation*}
		Since $p^{-1/(p-1)}\le 1$ for all $p\ge 1$, this shows $u\le 2(c+1)$ as claimed.
		By combining this with \eqref{eqn:bdd-nonneg}, we get $d(x_N,o) \le r$ which contradicts our assumption.
		This shows that the result holds with $B:=B_r(o)$.
	\end{proof}
	
	\begin{corollary}\label{cor:closed-bdd}
		For any metric space $(X,d)$, any $p\ge 1$, and any $\mu\in \Pcal_{p-1}(X)$, the set $M_p(\mu)\subseteq X$ is closed and bounded.
	\end{corollary}
	
	\begin{proof}
		By definition we have
		\begin{equation*}
			\obj_p(\mu) = \bigcap_{x'\in X}\{x\in X: W_p(\mu,x,x')\le 0\}.
		\end{equation*}
		From the continuity of $W_p$ \cite[Lemma~2.7]{EvansJaffeFrechet} we see that $\{x\in X: W_p(\mu,x,x')\le 0\}$ is closed for each $x'\in X$, so it follows that $\obj_p(\mu)$, being an intersection of closed sets, is closed.
		Also, from Lemma~\ref{lem:obj-bdd} we see that $\obj_p(\mu)$ is bounded.
	\end{proof}
	
	The natural next step in many existing works is to simply assume that $(X,d)$ is a Heine-Borel space so that Corollary~\ref{cor:closed-bdd} would immediately imply that $\obj_p$ is compact.
	In this work we will take a different route; instead, we will deduce compactness from our main result, which we prove now:
	
	\begin{theorem}\label{thm:main}
		Consider any separable metric space $(X,d)$ admitting a weak convergence and any $p\ge 1$, suppose that $\{\mu_n\}_{n\in\N}$ and $\mu$ in $\Pcal_{p-1}(X)$ have $\mu_n\to\mu$ in $\weak^{p-1}$, and suppose that non-negative real numbers $\{\varepsilon_n\}_{n\in\N}$ satisfy $\varepsilon_n\to 0$.
		Then, for any $\{x_n\}_{n\in\N}$ in $X$ with $x_n\in \obj_p(\mu_n;\varepsilon_n)$ for all $n\in\N$, the sequence $\{x_n\}_{n\in\N}$ is relatively $d$-compact and any subsequential limit thereof must lie in $\obj_p(\mu)$.
	\end{theorem}
	
	\begin{proof}
		Let $\{x_n\}_{n\in\N}$ be any sequence in $X$ with $x_n\in \obj_p(\mu_n;\varepsilon_n)$ for all $n\in\N$.
		Then, use Lemma~\ref{lem:obj-bdd} to get that $\{x_n\}_{n\in\N}$ is bounded.
		Now let $w$ denote a weak convergence $(X,d)$, and use (W1) to get that there exists a subsequence $\{n_k\}_{k\in\N}$ and a point $x\in X$ with $x_{n_k}\to x$ in $w$.
		It only remains to show that $x\in \obj_{p}(\mu)$ and that we have $x_{n_k}\to x$ in $d$.
		
		First let us show $x\in \obj_{p}(\mu)$, and we begin by fixing an arbitrary $o\in X$.
		Observe that by \eqref{eqn:renorm} and \eqref{eqn:ineq-1}, we can bound, for all $y\in X$ and $k\in\N$:
		\begin{align*}
			&d^p(x_{n_k},y)-d^p(o,y) \\
			&\ge -pd(x_{n_k},o)(d^{p-1}(x_{n_k},y)+d^{p-1}(o,y)) \\
			&\ge -pd(x_{n_k},o)(c_{p-1}(d^{p-1}(x_{n_k},o)+d^{p-1}(o,y))+d^{p-1}(o,y)) \\
			&\ge -pc_{p-1}d^p(x_{n_k},o) - p(c_{p-1}+1)d^{p-1}(o,y).
		\end{align*}
		Since $\{x_n\}_{n\in\N}$ is bounded and $\mu\in \Pcal_{p-1}(X)$, this implies that the functions $y\mapsto d^p(x_{n_k},y)-d^p(o,y)$ posess a $\mu$-integrable lower bound, uniformly in $k\in\N$.
		This means we can apply (W2) and Fatou's lemma, and the continuity of $W_p$ \cite[Lemma~2.7]{EvansJaffeFrechet}, to get, for arbitrary $x'\in X$:
		\begin{equation}\label{eqn:large-ineq}
			\begin{split}
				W_p(\mu,x,o) &= \int_{X}(d^p(x,y)-d^p(o,y))\diff \mu(y) \\
				&\le \int_{X}\liminf_{k\to\infty}(d^p(x_{n_k},y)-d^p(o,y))\diff \mu(y) \\
				&\le \liminf_{k\to\infty}\int_{X}(d^p(x_{n_k},y)-d^p(o,y))\diff \mu(y) \\
				&= \liminf_{k\to\infty}W_p(\mu,x_{n_k},o) \\
				&= \liminf_{k\to\infty}W_p(\mu_{n_k},x_{n_k},o) \\
				&\le \liminf_{k\to\infty}(W_p(\mu_{n_k},x',o) + \varepsilon_{n_k} )\\
				&= W_p(\mu, x', o)
			\end{split}
		\end{equation}
		In particular, by taking the infimum over all $x'\in X$, we find $x\in \obj_p(\mu)$.
		
		Towards showing $x_{n_k}\to x$ in $d$, we now make a short digression.
		For each subsequence $K = \{k_j\}_{j\in\N}$, we consider the set
		\begin{equation*}
			A_K := \left\{y\in X: \liminf_{j\to\infty} d(x_{n_{k_j}},y) = d(x,y) \right\}.
		\end{equation*}
		It is straightforward to see that each $A_K$ is a $d$-closed subset of $X$:
		If $\{y_\ell\}_{\ell\in\N}$ in $A_K$ have $y_\ell\to y$ in $d$ for some $y\in X$, then by the triangle inequality, we have
		\begin{align*}
			&\left|\liminf_{j\to\infty}d(x_{n_{k_j}},y)-d(x,y)\right|  \\
			&= 		\left|\liminf_{j\to\infty}d(x_{n_{k_j}},y)-\liminf_{j\to\infty}d(x_{n_{k_j}},y_\ell)+d(x,y_\ell)-d(x,y)\right|  \\
			&\le	\limsup_{j\to\infty}\left|d(x_{n_{k_j}},y)-d(x_{n_{k_j}},y_\ell)\right|+\left|d(x,y_\ell)-d(x,y)\right|  \\
			&\le 2d(y_\ell, y).
		\end{align*}
		Taking $\ell\to\infty$, we get $y\in A_K$ as needed.
		Moreover, we claim that each $A_K$ satisfies $\mu(A_K) = 1$.
		To show this, we use an argument identical to \eqref{eqn:large-ineq} above to get
		\begin{equation*}
			\begin{split}
				W_p(\mu,x,o) &= \int_{X}(d^p(x,y)-d^p(o,y))\diff \mu(y) \\
				&\le \int_{X}\liminf_{j\to\infty}(d^p(x_{n_{k_j}},y)-d^p(o,y))\diff \mu(y) \le \val_p(\mu,o).
			\end{split}
		\end{equation*}
		Since $\val_p(\mu,o)$ is defined as a minimum, this implies that the inequalities in the preceding display are actually equalities.
		In particular, we have
		\begin{equation*}
			\int_{X}(d^p(x,y)-d^p(o,y))\diff\mu(y) = \int_{X}\liminf_{j\to\infty}(d^p(x_{k_j},y)-d^p(o,y))\diff\mu(y).
		\end{equation*}
		Combining this with (W2) shows
		\begin{equation*}
			d^p(x,y)-d^p(o,y) = \liminf_{j\to\infty}(d^p(x_{k_j},y)-d^p(o,y))
		\end{equation*}
		for $\mu$-almost all $y\in X$, so rearranging gives $\mu(A_K) = 1$.
		
		Returning to the main proof, we observe that $\mu$ is a Borel measure on a second-countable topological space and that $\{A_K\}_K$ is an (arbitrary) intersection of closed sets of full $\mu$-measure.
		Consequently, the set $A:= \bigcap_{K}A_K$ satisfies $\mu(A) = 1$.
		Since $\mu(A) = 1$ implies that $A$ is non-empty, we can select an arbitrary $y\in A$ and we can then select a subsequence $\{k_j\}_{j\in\N}$ such that
		\begin{equation*}
			\limsup_{k\to\infty}d(x_{n_k},y) = \lim_{j\to\infty}d(x_{n_{k_j}},y).
		\end{equation*}
		Finally, use $y\in A$ and (W2) to get:
		\begin{align*}
			\limsup_{k\to\infty}d(x_{n_k},y) &= \lim_{j\to\infty}d(x_{n_{k_j}},y) \\
			&= \liminf_{j\to\infty}d(x_{n_{k_j}},y) \\
			&= d(x,y) \\
			&\le \liminf_{k\to\infty}d(x_{n_k},y).
		\end{align*}
		In other words, we have shown $d(x_{n_k},y)\to d(x,y)$ as $k\to\infty$.
		Therefore, (W3) implies $x_{n_k}\to x$ in $d$.
		This completes the proof.
	\end{proof}
	
	\begin{remark}\label{rem:thm-without-W3}
		Observe that (W3) was only used in the very last step of the proof of Theorem~\ref{thm:main}.
		Consequently the following conclusion holds if $(X,d)$ is a separable metric space admitting a convergence $w$ satisfying (W1) and (W2):
		If $\mu_n\to \mu$ in $\weak^{p-1}$ and $x_n\in \obj_p(\mu_n;\varepsilon_n)$ for all $n\in\N$, then $\{x_n\}_{n\in\N}$ is relatively $w$-compact and any subsequential limit thereof must lie in $M_p(\mu)$.
		In particular, Fr\'echet mean sets are non-empty, provided that there exists a convergence satisfying (W1) and (W2).
		We believe that weaker result may still be useful in some applications.
	\end{remark}
	
	Now we have some consequences of this main result.
	The statement of Theorem~\ref{thm:main} is given in the form of $\Gamma$-convergence, but it is often useful to use some of its more concrete consequences, as in the following.
	
	\begin{corollary}\label{cor:main-cpt}
		Consider any separable metric space $(X,d)$ admitting a weak convergence, and any $p\ge 1$
		For any $\mu\in \Pcal_{p-1}(X)$, the set $\obj_p(\mu)$ is non-empty and $d$-compact.
	\end{corollary}
	
	\begin{proof}
		To see that $M_p(\mu)$ is non-empty, set $\varepsilon_n:=1/n$ and $\mu_n:=\mu$ for $n\in\N$, and let $\{x_n\}_{n\in\N}$ be an arbitrary sequence in $\obj_p(\mu;\varepsilon_n)$.
		Theorem~\ref{thm:main} yields some $\{n_k\}_{k\in\N}$ and $x\in \obj_p(\mu)$ with $x_{n_k}\to x$ in $d$, so of course $M_p(\mu)$ is non-empty.
		To see that $M_p(\mu)$ is $d$-compact, set $\varepsilon_n:=0$ and $\mu_n:=\mu$ for $n\in\N$, and let $\{x_n\}_{n\in\N}$ be an arbitrary sequence in $\obj_p(\mu)$.
		Then Theorem~\ref{thm:main} yields some $\{n_k\}_{k\in\N}$ and $x\in \obj_p(\mu)$ with $x_{n_k}\to x$ in $d$, so $M_p(\mu)$ is sequentially compact.
		Since $M_p(\mu)$ is a subset of a metric space, this implies that it is compact.
	\end{proof}
	
	\begin{corollary}\label{cor:main-cty}
		Consider any separable metric space $(X,d)$ admitting a weak convergence, and any $p \ge 1$.
		If $\{\mu_n\}_{n\in\N}$ and $\mu$ in $\Pcal_{p-1}(X)$ have $\mu_n\to\mu$ in $\weak^{p-1}$, then
		\begin{equation*}
			\max_{x_n\in \obj_p(\mu_n)}\min_{x\in \obj_p(\mu)}d(x_n,x)\to 0
		\end{equation*}
		as $n\to\infty$.
	\end{corollary}
	
	\begin{proof}
		In order to show that
		\begin{equation*}
			\max_{x_n\in \obj_p(\mu_n)}\min_{x\in \obj_p(\mu)}d(x_n,x)
		\end{equation*}
		converges to 0, it suffices to show that every subsequence admits a further subsequence converging to 0.
		So, we let $\{n_k\}_{k\in\N}$ be arbitrary, and we use the compactness of Corollary~\ref{cor:main-cpt} and the continuity of $d$ to get, for each $k\in\N$, a point $x_k\in \obj_p(\mu_{n_k})$ with
		\begin{equation*}
			\min_{x\in\obj_p(\mu)}d(x_k,x) = \max_{x_{n_k}\in \obj_p(\mu_{n_k})}\min_{x\in \obj_p(\mu)}d(x_{n_k},x).
		\end{equation*}
		Now we use Theorem~\ref{thm:main} with $\varepsilon_n:= 0$ for all $n\in\N$ to get a subsequence $\{k_j\}_{j\in\N}$ and a point $x'\in \obj_p(\mu)$ with $x_{k_j}\to x'$ in $d$.
		By construction we have
		\begin{equation*}
			\max_{x_{n_k}\in \obj_p(\mu_{n_k})}\min_{x\in \obj_p(\mu)}d(x_{n_k},x) = \min_{x\in\obj_p(\mu)}d(x_k,x) \le d(x_{k_j},x')\to 0,
		\end{equation*}
		which proves the claim.
	\end{proof}
	
	Lastly, we show that these results imply a bona fide continuity result for the Fr\'echet mean, interpreted as a set-valued function.
	To state this, we write $\cpt(X)$ for the set of non-empty compact subsets of $X$.
	(We emphasize that these are $d$-compact, not $w$-compact.)
	Then, we write
	\begin{equation*}
		\dvechaus(S,S'):= \max_{x\in S}\min_{x'\in S'}d(x,x')
	\end{equation*}
	for the \textit{one-sided Hausdorff distance} between $S,S'\in\cpt(X)$.
	Although $\dvechaus$ is not a bona fide metric, it has many useful properties which can be easily shown.
	For example, $S,S'\in\cpt(X)$ have $\dvechaus(S,S') = 0$ if and only if $S\subseteq S'$, and all $S,S',S''\in\cpt(X)$ satisfy $\dvechaus(S,S'') \le \dvechaus(S,S') + \dvechaus(S',S'')$.
	Towards constructing a topology on $\cpt(X)$ induced by $\dvechaus$, we define $\vec B_r(S):=\{S'\in \cpt(X): \dvechaus(S',S)<r\}$, for each $S\in\cpt(X)$ and $r\ge0$, which we think of as the $\dvechaus$-ball of radius $r$ centered at $S$.
	Then we have the following:
	
	\begin{lemma}\label{lem:basis}
		For any metric space $(X,d)$, the collection $\{\vec B_r(S):S\in\cpt(X),r\ge 0\}$ is a basis on $\cpt(X)$.
	\end{lemma}
	
	\begin{proof}
		It suffices to show that if $S\in \vec B_{r_1}(S_1)\cap \vec B_{r_2}(S_2)$, then there exists $r\ge 0$ such that $S\in \vec B_r(S) \subseteq\vec B_{r_1}(S_1)\cap \vec B_{r_2}(S_2)$.
		Indeed, we have $\dvechaus(S,S_1)<r_1$ and $\dvechaus(S,S_2)<r_2$, so we set
		\begin{equation*}
			r:=\min\left\{r_1-\dvechaus(S,S_1),r_2-\dvechaus(S,S_2)\right\},
		\end{equation*}
		By construction we have $r>0$, so it follows that $S\in \vec B_r(S)$.
		To see that $\vec B_{r}(S)\subseteq \vec B_{r_1}(S_1)$, simply use the triangle inequality for $\dvechaus$ to see that $S\in \vec B_r(S)$ must satisfy
		\begin{align*}
			\dvechaus(S',S_1) &\le \dvechaus(S',S)+\dvechaus(S,S_1) \\
			&< r +\dvechaus(S,S_1) \\
			&\le r_1 -  \dvechaus(S,S_1) +\dvechaus(S,S_1) = r_1.
		\end{align*}
		This shows $\vec B_{r}(S)\subseteq \vec B_{r_1}(S_1)$, and of course the same argument shows $\vec B_{r}(S)\subseteq \vec B_{r_2}(S_2)$.
		We have now shown $S\in \vec B_r(S)\subseteq \vec B_{r_1}(S_1)\cap \vec B_{r_2}(S_2)$, so the result is proved.
	\end{proof}
	
	By the preceding result, there exists a topology on $\cpt(X)$ generated by $\{\vec B_r(S): S\in\cpt(X),r\ge 0\}$.
	Note, however, that this topology is quite degenerate; it is $T_0$ but not $T_1$.
	(In particular, it is not Hausdorff.)
	By a slight abuse of notation, we write $\dvechaus$ for this topology.
	At last we get the following bona fide continuity result.
	
	\begin{corollary}\label{cor:second-cty}
		If $(X,d)$ is a separable metric space admitting a weak convergence, then for all $p\ge 1$ the map $M_p:(\Pcal_{p-1}(X),\weak^{p-1})\to (\cpt(X),\dvechaus\,)$ is continuous.
	\end{corollary}
	
	\begin{proof}
		First note by Corollary~\ref{cor:main-cpt} that $M_p$ indeed takes values in $\cpt(X)$.
		Now use Lemma~\ref{lem:basis} to see that it suffices to show that $M_p^{-1}(\vec B_r(S))\subseteq \Pcal_{p-1}(X)$ is $\weak^{p-1}$-open for all $S\in\cpt(X)$ and $r\ge 0$.
		Indeed, suppose that we have $\mu\in M_p^{-1}(\vec B_r(S))$, which is equivalent to $\dvechaus(M_p(\mu),S) < r$.
		Now set $\varepsilon:=r-\dvechaus(M_p(\mu),S) >0$, and use Corollary~\ref{cor:main-cty} to get some $\weak^{p-1}$-open set $U\subseteq \Pcal_{p-1}(X)$ such that $\nu\in U$ implies $\dvechaus(M_p(\nu),M_p(\mu))<\varepsilon$.
		In particular, we can combine this with the triangle inequality for $\dvechaus$ to show:
		\begin{align*}
			\dvechaus(M_p(\nu),S) &\le \dvechaus(M_p(\nu),M_{p}(\mu))+\dvechaus(M_p(\mu),S) \\
			&< \varepsilon+\dvechaus(M_p(\mu),S) \\
			&= r - \dvechaus(M_p(\mu),S) +\dvechaus(M_p(\mu),S) = r.
		\end{align*}
		This shows $U\subseteq M_p^{-1}(\vec B_r(S))$ and hence that $M_p^{-1}(\vec B_r(S))$ is $\weak^{p-1}$-open.
		This completes the proof.
	\end{proof}
	
	In some cases, one has a subset $\mathcal{P}\subseteq\Pcal_{p-1}(X)$ such that $M_p(\mu)$ is unique for $\mu\in\Pcal$; in this setting, one can ignore all set-valued considerations and simply conclude that the map $M_p:(\Pcal,\weak^{p-1})\to (X,d)$ is continuous.
	(For example, if $(X,d)$ is a Hadamard space then one can take $\Pcal:=\Pcal_{p-1}(X)$, and if $(X,d)$ is the Wasserstein space $(\Pcal_q(\R^k,W_q)$ then one can take $\Pcal:=\Pcal_{p-1}(\Pcal_q^{\textnormal{ac}}(\R^k))$, the space of laws of random measures which have a density.)
	
	\section{Probabilistic Consequences}\label{sec:prob}
	
	This section contains the main probabilistic results of interest, and it proves them easily as a consequence of the main theorem above.
	Concretely, for suitable random variables taking values in a separable metric space admitting a weak convergence, we prove: a strong law of large numbers (Corollary~\ref{cor:SLLN}), a pointwise ergodic theorem (Corollary~\ref{cor:ergodic}), and a large deviations principle (Corollary~\ref{cor:LDP}).
	
	Throughout this section, we let $(X,d)$ denote a separable metric space admitting a weak convergence and we fix $p\ge 1$.
	
	\begin{corollary}[strong law of large numbers]\label{cor:SLLN}
		Suppose that $(\Omega,\F,\P)$ is a probability space, with expectation $\E$, supporting an independent, identically-distributed sequence $Y_1,Y_2,\ldots$ of $X$-valued random variables with common distribution $\mu$, and let $p\ge 1$.
		Also write $\bar \mu_n:=\frac{1}{n}\sum_{i=1}^{n}\delta_{Y_i}$ for the emirical meausure of the first $n\in\N$ points.
		If we have $\int_{X}d^{p-1}(x,y)\diff \mu(y) = \E[d^{p-1}(x,Y_1)]<\infty$ for some (equivalently, all) $x\in X$, then we have
		\begin{equation*}
			\max_{\bar x_n \in M_p(\bar \mu_n)}\min_{x\in M_p(\mu)}d(\bar x_n,x)\to 0
		\end{equation*}
		holding $\P$-almost surely.
	\end{corollary}
	
	\begin{proof}
		By \cite{Varadarajan2}, we may let $\{\phi_k\}_{k\in\N}$ denote a sequence of bounded, continuous functions from $(X,d)$ to $\R$ such that Borel probability measures $\{\mu_n\}_{n\in\N}$ on $X$ have $\mu_n\to \mu$ in $\weak$ if and only if we have $\int_{X}\phi_k\diff \mu_n\to \int_{X}\phi_k\diff \mu$ as $n\to\infty$ for all $k\in\N$; thus, let us define the event
		\begin{equation*}
			C_k :=\left\{ \int_{X}\phi_k\diff \bar \mu_n\to \int_{X}\phi_k\diff \mu  \textnormal{ as } n\to\infty\right\}
		\end{equation*}
		for all $k\in\N$.
		We also define
		\begin{equation*}
			D :=\left\{ \int_{X}d^{p-1}(o,y)\diff \bar \mu_n\to \int_{X}d^{p-1}(o,y)\diff \mu  \textnormal{ as } n\to\infty\right\},
		\end{equation*}
		for arbitrary $o\in\N$, and we note by the well-known characterization of convergence in the Wasserstein topologies that we have $D\cap\bigcap_{k\in\N}C_k = \{\bar \mu_n\to \mu \textnormal{ in } \weak^{p-1}\}$.
		Of course, for each $k\in\N$ we have
		\begin{equation*}
			\P(C_k) = \P\left(\frac{1}{n}\sum_{i=1}^{n}\phi_k(Y_i)\to \int_{X}\phi_k\diff \mu\right) = 1
		\end{equation*}
		by the classical (real-valued) strong law of large numbers since $\phi_k$ is bounded.
		Similarly, the integrability assumption on $\mu$ implies
		\begin{equation*}
			\P(D) = \P\left(\frac{1}{n}\sum_{i=1}^{n}d^{p-1}(o,Y_i)\to \int_{X}d^{p-1}(o,y)\diff \mu(y)\right) = 1
		\end{equation*}
		by the classical (real-valued) strong law of large numbers.
		By taking the intersection of these events, we have shown $\P(\bar \mu_n\to \mu \textnormal{ in } \weak^{p-1}) = 1$, so the result follows from Corollary~\ref{cor:main-cty}.
	\end{proof}
	
	This result finally recovers the desired behavior for the strong law of large numbers, wherein 1 finite moment is sufficient for convergence of $M_2$.
	Additionally, it shows that no integrability is required in order to ge the convergence of $M_1$, which is a fundamental result for robust statistics in non-Euclidean spaces.
	These results allow us to reduce the moment conditions in the literature, where it is common to assume 2 finite moments in order to guarantee convergence of $M_2$. (See \cite{ManifoldsI, EvansJaffeFrechet,Schoetz2, WassersteinBarycenter, Huckemann, PanaretosSantoroBW}.) 
	Also, we empasize that the assumption of $p-1$ finite moments for the convergence of $M_p$ as above cannot, in general, be relaxed; this is because it is already known that the strong law of large numbers for $M_2$ in $\R$ fails without 1 finite moment.
	
	We also mention that this result, when combined with the results of Section~\ref{sec:weak}, yields new asymptotic theory for Fr\'echet means used in some infinite-dimensional settings of interest.
	For instance, we get a strong law of large numbers for Fr\'echet means in the space of unparameterized unregistered loops that we studied in Example~\ref{ex:loops}, which, despite being common tool in shape analysis and computational anatomy, has not been thoroughly studied from a probabilistic point of view.
	
	Lastly, we emphasize that the sort of strong convergence above holds in much more general settings:
	If $\bar \mu_n$ converges almost surely in $\weak^{p-1}$ to some $\mu$ (in fact, $\mu$ can be random) then $M_p(\bar \mu_n)$ converges to $M_p(\mu)$ in the sense above.
	So, for example, we may let $Y_1,Y_2,\ldots$ be a Harris-recurrent Marov chain, a stationary sequence, an exchangeable sequence, a 1-dependent (or $k$-dependent for fixed $k$) sequence, and more.
	
	\medskip
	
	The setting of stationary sequences above can be thought of as a problem in ergodic theory, and this deserves further attention.
	To make this discussion precise, we need to define a few notions.
	Recall that a locally compact group $G$ is called \textit{amenable} if it admits a sequence of compact subsets $\{F_n\}_{n\in\N}$ such that $m(F_n)<\infty$ for all $n\in\N$ and such that we have $m(gF_n\Delta F_n)/m(F_n)\to 0$ as $n\to\infty$ for all $g\in G$; here, $m$ is the left-invariant Haar measure of $G$, and $\{F_n\}_{n\in\N}$ is called a \textit{F{\o}lner sequence} of $G$.
	We say that a F{\o}lner sequence is \textit{tempered} if there exists a constant $C>0$ such that we have
	\begin{equation*}
		m\left(\bigcup_{1\le k < n}F_k^{-1}F_n\right)\le Cm(F_n)
	\end{equation*}
	for all $n\in\N$.
	It is known that every locally compact amenable group admits a tempered F{\o}lner sequence, and one can in fact ``thin out'' any given F{\o}lner sequence to form a tempered F{\o}lner sequence	\cite[Proposition~1.4]{Lindenstrauss}.
	Then we have the following:
	
	\begin{corollary}[pointwise ergodic theorem]\label{cor:ergodic}
		Suppose that $G$ is a locally compact amenable group endowed with a jointly measurable, measure-preserving action on a standard probability space $(\Omega,\F,\P)$, and let $Y:\Omega \to X$ be any measurable map.; we write $m$ for the left-invariant Haar measure of $G$, and we let denote $\{F_n\}_{n\in\N}$ any tempered F{\o}lner sequence of $G$.
		Also define $\mu_{Y,n}^{\omega}\in\Pcal(X)$ via
		\begin{equation*}
			\mu_{Y,n}^{\omega}:= \frac{1}{m(F_n)}\int_{F_n}\delta_{Y(g\cdot\omega)} \diff m(g)
		\end{equation*}
		for $\omega\in\Omega$.
		If we have $\int_{\Omega}d^p(x,Y(\omega))\diff \P(\omega)<\infty$ for some (equivalently, all) $x\in X$, then there exists a probability measure $\mu_{Y}^{\omega}\in\Pcal(X)$ for each $\omega\in\Omega$ such that $\mu_Y:\Omega\to\Pcal_{p-1}(X)$ is $G$-invariant and such that we have
		\begin{equation*}
			\max_{x_n \in M_p(\mu_{Y,n}^{\omega})}\min_{x\in M_p(\mu_{Y}^{\omega})}d(x_n,x)\to 0.
		\end{equation*}
		for $\P$-almost all $\omega\in\Omega$.
		Furthermore, if the action of $G$ is ergodic, then we have $\mu_{Y}^{\omega} = \P\circ Y^{-1}$ for $\P$-almost all $\omega\in\Omega$.
	\end{corollary}
	
	\begin{proof}
		Let $\{\phi_k\}_{k\in\N}$ be as in the proof Corollary~\ref{cor:SLLN}.
		By the Lindenstrauss pointwise ergodic theorem for locally compact amenable groups \cite[Theorem~1.2]{Lindenstrauss}, the set
		\begin{equation*}
			C_k :=\left\{ \omega\in\Omega:\lim_{n\to\infty}\int_{X}\phi_k(y)\diff \mu_{Y,n}^{\omega}(y) \textnormal{ exists } \right\}
		\end{equation*}
		satisfies $\P(C_k)=1$ for all $k\in\N$, hence the set $C:=\bigcap_{k\in\N}C_k$ satisfies $\P(C) = 1$.
		Now for $\omega\in C$, we define the measure $\mu_Y^{\omega}\in\Pcal(X)$ via
		\begin{equation*}
			\int_{X}\phi_k(y)\diff \mu_Y^{\omega}(y):=\lim_{n\to\infty}\int_{X}\phi_k(y)\diff \mu_{Y,n}^{\omega}(y)
		\end{equation*}
		for all $k\in\N$; observe that if the action of $G$ is ergodic, then for each $\omega\in C$ we have $\int_{X}\phi_k(y)\diff \mu_{Y,n}^{\omega}(y)\to \int_{\Omega}\phi_k(Y(\omega))\diff \P(\omega)$ as $n\to\infty$ hence $\mu_Y^{\omega} = \P\circ Y^{-1}$.
		By construction we have $\mu_{Y,n}^{\omega}\to \mu_Y^{\omega}$ in $\weak$ for each $\omega\in C$.
		Next, observe that the function $d^{p-1}(o,Y):\Omega\to\R$ lies in $L^1((\Omega,\F,\P);\R)$ by the integrability assumption for fixed $o\in X$; thus, another application of the Lindenstraus theorem \cite[Theorem~1.2]{Lindenstrauss} shows that we have
		\begin{equation*}
			\int_{X}d^{p-1}(o,y)\diff \mu_{Y,n}^{\omega}(y) \to \int_{X}d^{p-1}(o,y)\diff \mu_{Y}^{\omega}(y)
		\end{equation*}
		as $n\to\infty$, for all $\omega\in D$ where $D\in\F$ is some set satisfying $\P(D) = 1$.
		In particular, we have $\mu_{Y,n}^{\omega}\to \mu_Y^{\omega}$ in $\weak^{p-1}$ for each $\omega\in C\cap D$, and of course $\P(C\cap D) = 1$.
		Finally, we conclude by applying Corollary~\ref{cor:main-cty}.
	\end{proof}
	
	Let us discuss a few special cases of this result.
	First of all, if $G=\N$ and one takes $F_n = \{0,\ldots, n\}$ for $n\in\N$, then we can see that this includes Birkhoff's pointwise ergodic theorem, but in the general setting of Fr\'echet $p$-means in metric spaces admitting a weak convergence.
	Second, if $(X,d)$ is a Hadamard space and $p=2$, then this recovers and strengthens a result of Austin \cite{Austin} by reducing the assumption of 2 finite moments to an assumption of 1 finite moment; this is arguably more natural the results \cite{Navas, EsSahibHeinich} which obtain the minimal moment assumption by studying a different notion of barycenter.
	
	\medskip
	
	Finally, we focus on large deviations theory, for which we need to introduce some notation.
	For Borel probability measures $\mu,\nu$ on $(X,d)$, we write
	\begin{equation*}
		H(\nu\,|\,\mu) := \begin{cases}
			\int_X \log\left(\frac{\diff \nu}{\diff \mu}\right) \diff \nu &\textnormal{ if } \nu \ll \mu \\
			\infty &\textnormal{ else}
		\end{cases}
	\end{equation*}
	called the \textit{relative entropy} of $\nu$ from $\mu$.
	Then we get the following:
	
	\begin{corollary}[large deviations principle]\label{cor:LDP}
		Suppose that $(\Omega,\F,\P)$ is a probability space, with expectation $\E$, supporting an independent, identically-distributed sequence $Y_1,Y_2,\ldots$ of $X$-valued random variables with common distribution $\mu$, and let $p\ge 1$.
		Also write $\bar \mu_n:=\frac{1}{n}\sum_{i=1}^{n}\delta_{Y_i}$ for the emirical meausure of the first $n\in\N$ points.
		Moreover, let us assume that $M_p(\bar \mu_n)$ is a singleton almost surely for all $n\in\N$.
		If we have $\int_{X}\exp(\lambda d^{p-1}(x,y))\diff \mu(y) = \E[\exp(\lambda d^{p-1}(x,Y_1))]<\infty$ for all $\lambda>0$ and some (equivalently, all) $x\in X$, then $\{M_p(\bar \mu_n)\}_{n\in\N}$ satisfies a large deviations principle in $(X,d)$ with good rate function
		\begin{equation*}
			I_{p,\mu}(x) := \inf\{H(\nu\,|\,\mu): \nu\in\mathcal{P}_{p-1}(X), M_p(\nu) =x\}
		\end{equation*}
		for $x\in X$.
	\end{corollary}
	
	\begin{proof}
		By \cite[Theorem~1.1]{SanovWasserstein} and the assumption $\E[\exp(\lambda d^{p-1}(x,Y_1))]<\infty$ for some $\lambda>0$ and all $x\in X$, the random measures $\{\bar \mu_n\}_{n\in\N}$ satisfy a large deviations principle in $(\mathcal{P}_{p-1}(X),\weak^{p-1})$ with good rate function $H(\,\cdot \, | \, \mu)$.
		Also, Corollary~\ref{cor:second-cty} and the remarks thereafter guarantee that the map $M_p:(\Pcal_{p-1}(X),\weak^{p-1})\to (X,d)$ is continuous.
		Thus, the result follows from the contraction principle \cite[Theorem~4.2.1]{LargeDeviations}.
	\end{proof}
	
	Some remarks about this result are due.
	First of all, we note that, if the Fr\'echet means are not assumed to be unique, then one can still prove a large deviations upper bound for the random sets $\{M_p(\bar \mu_n)\}_{n\in\N}$ in the space $(\cpt(X),\dvechaus)$; however, we are not aware of how to prove the lower bound in general, since $(\cpt(X),\dvechaus)$ is not Hausdorff.
	
	Second, we note that the optimality of the moment assumption is rather important.
	If we were to require $\mu_n\to \mu$ in $\weak^p$ in Theorem~\ref{thm:main}, then we would require $\E[\exp(\lambda d^p(x,Y_1))]<\infty$ for all $\lambda>0$ and some (equivalently, all) $x\in X$ in Corollary~\ref{cor:LDP}.
	However, if $p\ge 2$, then the latter condition is too strong to even cover the case that $d(x,Y_1)$ has sub-Gaussian tails.
	
	Third, let us remark that the rate function appearing above is not so useful in its current form.
	Indeed, it is written as the solution to a relative entropy projection problem under a constraint on the Fr\'echet mean.
	Interestingly, recent work \cite{JaffeSantoroLDP} has shown, in the case of the Wasserstein space, Bures-Wasserstein space, and Riemannian manifolds, that one can massage this rate function into a geometrically meaningful and numerically computable form.
	
	\nocite{*}
	\bibliography{FrechetMeansInfDim}

\begin{thebibliography}{100}

\bibitem{Afsari}
{\sc B.~Afsari}, {\em Riemannian {$L^p$} center of mass: existence, uniqueness,
  and convexity}, Proc. Amer. Math. Soc., 139 (2011), pp.~655--673.

\bibitem{AguehCarlier}
{\sc M.~Agueh and G.~Carlier}, {\em Barycenters in the {W}asserstein space},
  SIAM J. Math. Anal., 43 (2011), pp.~904--924.

\bibitem{agueh2017vers}
\leavevmode\vrule height 2pt depth -1.6pt width 23pt, {\em Vers un th\'eor\`eme
  de la limite centrale dans l'espace de {W}asserstein?}, C. R. Math. Acad.
  Sci. Paris, 355 (2017), pp.~812--818.

\bibitem{AGP}
{\sc A.~Ahidar-Coutrix, T.~Le~Gouic, and Q.~Paris}, {\em Convergence rates for
  empirical barycenters in metric spaces: curvature, convexity and extendable
  geodesics}, Probab. Theory Related Fields, 177 (2020), pp.~323--368.

\bibitem{AlexandrinoBettiol}
{\sc M.~M. Alexandrino, R.~G. Bettiol, et~al.}, {\em Lie groups and geometric
  aspects of isometric actions}, vol.~82, Springer, 2015.

\bibitem{AmbrosioGigliSavare}
{\sc L.~Ambrosio, N.~Gigli, and G.~Savaré}, {\em Gradient Flows in Metric
  Spaces and in the Space of Probability Measures}, Lectures in Mathematics
  {{ETH Zürich}}, {Birkhäuser}, 2. ed~ed.
\newblock OCLC: 254181287.

\bibitem{Austin}
{\sc T.~Austin}, {\em A {$\rm CAT(0)$}-valued pointwise ergodic theorem}, J.
  Topol. Anal., 3 (2011), pp.~145--152.

\bibitem{BacakOptimization}
{\sc M.~Ba{\v{c}}{\'a}k}, {\em Convex analysis and optimization in {H}adamard
  spaces}, vol.~22 of De Gruyter Series in Nonlinear Analysis and Applications,
  De Gruyter, Berlin, 2014.

\bibitem{TreeCLT}
{\sc D.~Barden, H.~Le, and M.~Owen}, {\em {Central limit theorems for
  Fr{\'e}chet means in the space of phylogenetic trees}}, Electronic Journal of
  Probability, 18 (2013), pp.~1 -- 25.

\bibitem{ShapeSpaces}
{\sc M.~Bauer, M.~Bruveris, and P.~W. Michor}, {\em Overview of the geometries
  of shape spaces and diffeomorphism groups}, J. Math. Imaging Vision, 50
  (2014), pp.~60--97.

\bibitem{ManifoldsI}
{\sc R.~Bhattacharya and V.~Patrangenaru}, {\em Large sample theory of
  intrinsic and extrinsic sample means on manifolds. {I}}, Ann. Statist., 31
  (2003), pp.~1--29.

\bibitem{ManifoldsII}
\leavevmode\vrule height 2pt depth -1.6pt width 23pt, {\em Large sample theory
  of intrinsic and extrinsic sample means on manifolds. {II}}, Ann. Statist.,
  33 (2005), pp.~1225--1259.

\bibitem{BHV}
{\sc L.~J. Billera, S.~P. Holmes, and K.~Vogtmann}, {\em Geometry of the space
  of phylogenetic trees}, Adv. in Appl. Math., 27 (2001), pp.~733--767.

\bibitem{FrechetRelaxation}
{\sc M.~Blanchard and A.~Q. Jaffe}, {\em {Fr{\'e}chet mean set estimation in
  the Hausdorff metric, via relaxation}}, Bernoulli, 31 (2025), pp.~432 -- 456.

\bibitem{BourbakiII}
{\sc N.~Bourbaki and S.~Berberian}, {\em Integration II: Chapters 7--9},
  Actualit{\'e}s scientifiques et industrielles, Springer Berlin Heidelberg,
  2004.

\bibitem{Brezis}
{\sc H.~Brezis}, {\em Functional analysis, {S}obolev spaces and partial
  differential equations}, Universitext, Springer, New York, 2011.

\bibitem{BridsonHaefliger}
{\sc M.~R. Bridson and A.~Haefliger}, {\em Metric spaces of non-positive
  curvature}, vol.~319 of Grundlehren der mathematischen Wissenschaften
  [Fundamental Principles of Mathematical Sciences], Springer-Verlag, Berlin,
  1999.

\bibitem{FermatWeberRevisited}
{\sc J.~Brimberg}, {\em The {F}ermat-{W}eber location problem revisited}, Math.
  Programming, 71 (1995), pp.~71--76.

\bibitem{Brunel}
{\sc V.-E. Brunel and J.~Serres}, {\em Concentration of empirical barycenters
  in metric spaces}, in Proceedings of The 35th International Conference on
  Algorithmic Learning Theory, C.~Vernade and D.~Hsu, eds., vol.~237 of
  Proceedings of Machine Learning Research, PMLR, 25--28 Feb 2024,
  pp.~337--361.

\bibitem{BruverisMichorMumford}
{\sc M.~Bruveris, P.~W. Michor, and D.~Mumford}, {\em Geodesic completeness for
  {S}obolev metrics on the space of immersed plane curves}, Forum Math. Sigma,
  2 (2014), pp.~Paper No. e19, 38.

\bibitem{BBI}
{\sc D.~Burago, Y.~Burago, and S.~Ivanov}, {\em A course in metric geometry},
  vol.~33 of Graduate Studies in Mathematics, American Mathematical Society,
  Providence, RI, 2001.

\bibitem{PersistenceUniqueness}
{\sc Y.~Cao and A.~Monod}, {\em A geometric condition for uniqueness of
  {F}r\'{e}chet means of persistence diagrams}, 2022.

\bibitem{WassersteinMedian}
{\sc G.~Carlier, E.~Chenchene, and K.~Eichinger}, {\em Wasserstein {M}edians:
  {R}obustness, {PDE} {C}haracterization, and {N}umerics}, SIAM J. Math. Anal.,
  56 (2024), pp.~6483--6520.

\bibitem{FrechetCircle}
{\sc F.~Cazals, B.~Delmas, and T.~O'Donnell}, {\em Fr\'echet mean and
  {$p$}-mean on the unit circle: decidability, algorithm, and applications to
  clustering on the flat torus}, in 19th {I}nternational {S}ymposium on
  {E}xperimental {A}lgorithms, vol.~190 of LIPIcs. Leibniz Int. Proc. Inform.,
  Schloss Dagstuhl. Leibniz-Zent. Inform., Wadern, 2021, pp.~Art. No. 15, 16.

\bibitem{ChandrasekaranTamir}
{\sc R.~Chandrasekaran and A.~Tamir}, {\em Open questions concerning
  weiszfeld's algorithm for the fermat-weber location problem}, Mathematical
  Programming, 44 (1989), pp.~293--295.

\bibitem{WFR}
{\sc L.~Chizat, G.~Peyr\'e, B.~Schmitzer, and F.-X. Vialard}, {\em An
  interpolating distance between optimal transport and {F}isher-{R}ao metrics},
  Found. Comput. Math., 18 (2018), pp.~1--44.

\bibitem{CuestaMatran}
{\sc J.~A. Cuesta and C.~Matr\'{a}n}, {\em The strong law of large numbers for
  {$k$}-means and best possible nets of {B}anach valued random variables},
  Probab. Theory Related Fields, 78 (1988), pp.~523--534.

\bibitem{LargeDeviations}
{\sc A.~Dembo and O.~Zeitouni}, {\em Large deviations techniques and
  applications}, vol.~38 of Stochastic Modelling and Applied Probability,
  Springer-Verlag, Berlin, 2010.
\newblock Corrected reprint of the second (1998) edition.

\bibitem{DivolLacombe}
{\sc V.~Divol and T.~Lacombe}, {\em Understanding the topology and the geometry
  of the space of persistence diagrams via optimal partial transport}, J. Appl.
  Comput. Topol., 5 (2021), pp.~1--53.

\bibitem{dryden2009non}
{\sc I.~L. Dryden, A.~Koloydenko, and D.~Zhou}, {\em Non-{E}uclidean statistics
  for covariance matrices, with applications to diffusion tensor imaging}, Ann.
  Appl. Stat., 3 (2009), pp.~1102--1123.

\bibitem{DepthProfile}
{\sc P.~Dubey, Y.~Chen, and H.-G. M\"uller}, {\em Metric statistics:
  exploration and inference for random objects with distance profiles}, Ann.
  Statist., 52 (2024), pp.~757--792.

\bibitem{FrechetSphere}
{\sc G.~Eichfelder, T.~Hotz, and J.~Wieditz}, {\em An algorithm for computing
  {F}r\'echet means on the sphere}, Optim. Lett., 13 (2019), pp.~1523--1533.

\bibitem{EsSahibHeinich}
{\sc A.~Es-Sahib and H.~Heinich}, {\em Barycentre canonique pour un espace
  m\'etrique \`a{} courbure n\'egative}, in S\'eminaire de {P}robabilit\'es,
  {XXXIII}, vol.~1709 of Lecture Notes in Math., Springer, Berlin, 1999,
  pp.~355--370.

\bibitem{EspinolaFernandezLeon}
{\sc R.~Esp\'inola and A.~Fern\'andez-Le\'on}, {\em {${\rm CAT}(k)$}-spaces,
  weak convergence and fixed points}, J. Math. Anal. Appl., 353 (2009),
  pp.~410--427.

\bibitem{EvansJaffeFrechet}
{\sc S.~N. Evans and A.~Q. Jaffe}, {\em Limit theorems for {F}r\'echet mean
  sets}, Bernoulli, 30 (2024), pp.~419--447.

\bibitem{EvansLidman}
{\sc S.~N. Evans and T.~Lidman}, {\em Expectation, conditional expectation and
  martingales in local fields}, Electron. J. Probab., 12 (2007), pp.~no. 17,
  498--515.

\bibitem{BookCLTs}
{\sc T.~Hotz, S.~Huckemann, H.~Le, J.~S. Marron, J.~C. Mattingly, E.~Miller,
  J.~Nolen, M.~Owen, V.~Patrangenaru, and S.~Skwerer}, {\em Sticky central
  limit theorems on open books}, Ann. Appl. Probab., 23 (2013), pp.~2238--2258.

\bibitem{Huckemann}
{\sc S.~F. Huckemann}, {\em Intrinsic inference on the mean geodesic of planar
  shapes and tree discrimination by leaf growth}, Ann. Statist., 39 (2011),
  pp.~1098--1124.

\bibitem{Huckemann2015}
{\sc S.~F. Huckemann}, {\em (Semi-)Intrinsic Statistical Analysis on
  Non-Euclidean Spaces}, Springer International Publishing, Cham, 2015,
  pp.~103--118.

\bibitem{MMT_I}
{\sc D.~T. J.~C.~Mattingly, E.~Miller}, {\em A central limit theorem for random
  tangent fields on stratified spaces}, 2023.
\newblock Pre-print available at https://arxiv.org/abs/2311.09454.

\bibitem{MMT_III}
\leavevmode\vrule height 2pt depth -1.6pt width 23pt, {\em Central limit
  theorems for fr\'echet means on stratified spaces}, 2023.
\newblock Pre-print available at https://arxiv.org/abs/2311.09455.

\bibitem{MMT_II}
\leavevmode\vrule height 2pt depth -1.6pt width 23pt, {\em Geometry of measures
  on smoothly stratified metric spaces}, 2023.
\newblock Pre-print available at https://arxiv.org/abs/2311.09453.

\bibitem{MMT_IV}
\leavevmode\vrule height 2pt depth -1.6pt width 23pt, {\em Shadow geometry at
  singular points of cat($k$) spaces}, 2023.
\newblock Pre-print available at https://arxiv.org/abs/2311.09451.

\bibitem{JaffeClustering}
{\sc A.~Q. Jaffe}, {\em Strong consistency for a class of adaptive clustering
  procedures}, 2022.
\newblock Pre-print available at https://arxiv.org/abs/2202.13423.

\bibitem{JaffeSantoroLDP}
{\sc A.~Q. Jaffe and L.~V. Santoro}, {\em Large deviations principle for
  {B}ures-{W}asserstein barycenters}, 2024.
\newblock Pre-print available at https://arxiv.org/abs/2409.11384.

\bibitem{Jakubowski}
{\sc A.~Jakubowski}, {\em The almost sure {Skorokhod} representation for
  subsequences in nonmetric spaces}, Teor. Veroyatn. Primen., 42 (1997),
  pp.~209--216.

\bibitem{JostWeak}
{\sc J.~Jost}, {\em Equilibrium maps between metric spaces}, Calc. Var. Partial
  Differential Equations, 2 (1994), pp.~173--204.

\bibitem{Jost}
\leavevmode\vrule height 2pt depth -1.6pt width 23pt, {\em Nonpositive
  curvature: geometric and analytic aspects}, Lectures in Mathematics ETH
  Z\"{u}rich, Birkh\"{a}user Verlag, Basel, 1997.

\bibitem{Kallenberg}
{\sc O.~Kallenberg}, {\em Foundations of modern probability}, Probability and
  its Applications (New York), Springer-Verlag, New York, second~ed., 2002.

\bibitem{Kaminska}
{\sc A.~Kami\'nska}, {\em On uniform convexity of {O}rlicz spaces}, Nederl.
  Akad. Wetensch. Indag. Math., 44 (1982), pp.~27--36.

\bibitem{Karcher}
{\sc H.~Karcher}, {\em Riemannian center of mass and mollifier smoothing},
  Comm. Pure Appl. Math., 30 (1977), pp.~509--541.

\bibitem{Kell}
{\sc M.~Kell}, {\em Uniformly convex metric spaces}, Anal. Geom. Metr. Spaces,
  2 (2014), pp.~359--380.

\bibitem{Kendall}
{\sc W.~S. Kendall}, {\em Probability, convexity, and harmonic maps with small
  image. {I}. {U}niqueness and fine existence}, Proc. London Math. Soc. (3), 61
  (1990), pp.~371--406.

\bibitem{KirkPanyanak}
{\sc W.~A. Kirk and B.~Panyanak}, {\em A concept of convergence in geodesic
  spaces}, Nonlinear Anal., 68 (2008), pp.~3689--3696.

\bibitem{UnlabeledGraphs}
{\sc E.~D. Kolaczyk, L.~Lin, S.~Rosenberg, J.~Walters, and J.~Xu}, {\em
  Averages of unlabeled networks: geometric characterization and asymptotic
  behavior}, Ann. Statist., 48 (2020), pp.~514--538.

\bibitem{Kroshnin}
{\sc A.~Kroshnin, V.~Spokoiny, and A.~Suvorikova}, {\em Statistical inference
  for {B}ures-{W}asserstein barycenters}, Ann. Appl. Probab., 31 (2021),
  pp.~1264--1298.

\bibitem{Kurtek}
{\sc S.~Kurtek, A.~Srivastava, E.~Klassen, and Z.~Ding}, {\em Statistical
  modeling of curves using shapes and related features}, J. Amer. Statist.
  Assoc., 107 (2012), pp.~1152--1165.

\bibitem{LeKume}
{\sc H.~Le and A.~Kume}, {\em The {F}r\'echet mean shape and the shape of the
  means}, Adv. in Appl. Probab., 32 (2000), pp.~101--113.

\bibitem{WassersteinBarycenter}
{\sc T.~Le~Gouic and J.-M. Loubes}, {\em Existence and consistency of
  {W}asserstein barycenters}, Probab. Theory Related Fields, 168 (2017),
  pp.~901--917.

\bibitem{WassersteinConvergence}
{\sc T.~Le~Gouic, Q.~Paris, P.~Rigollet, and A.~J. Stromme}, {\em Fast
  convergence of empirical barycenters in {A}lexandrov spaces and the
  {W}asserstein space}, J. Eur. Math. Soc. (JEMS), 25 (2023), pp.~2229--2250.

\bibitem{Lember}
{\sc J.~Lember}, {\em On minimizing sequences for {$k$}-centres}, J. Approx.
  Theory, 120 (2003), pp.~20--35.

\bibitem{Lim}
{\sc T.~C. Lim}, {\em Remarks on some fixed point theorems}, Proc. Amer. Math.
  Soc., 60 (1976), pp.~179--182 (1977).

\bibitem{Lindenstrauss}
{\sc E.~Lindenstrauss}, {\em Pointwise theorems for amenable groups}, Invent.
  Math., 146 (2001), pp.~259--295.

\bibitem{Lovasz}
{\sc L.~Lov\'{a}sz}, {\em Large networks and graph limits}, vol.~60 of American
  Mathematical Society Colloquium Publications, American Mathematical Society,
  Providence, RI, 2012.

\bibitem{LytchakPetrunin}
{\sc A.~Lytchak and A.~Petrunin}, {\em Weak topology on {${\rm CAT}(0)$}
  spaces}, Israel J. Math., 255 (2023), pp.~763--781.

\bibitem{Masarotto}
{\sc V.~Masarotto, V.~M. Panaretos, and Y.~Zemel}, {\em Procrustes metrics on
  covariance operators and optimal transportation of {G}aussian processes},
  Sankhya A, 81 (2019), pp.~172--213.

\bibitem{masarotto2022transportation}
\leavevmode\vrule height 2pt depth -1.6pt width 23pt, {\em Transportation-based
  functional {ANOVA} and {PCA} for covariance operators}, Electron. J. Stat.,
  18 (2024), pp.~1887--1916.

\bibitem{Megginson}
{\sc R.~E. Megginson}, {\em An Introduction to Banach Space Theory}, Graduate
  Texts in Mathematics, Springer-Verlag, New York, first~ed., 1998.

\bibitem{MichorMumford}
{\sc P.~W. Michor and D.~Mumford}, {\em Riemannian geometries on spaces of
  plane curves}, J. Eur. Math. Soc. (JEMS), 8 (2006), pp.~1--48.

\bibitem{MillerYounes}
{\sc M.~I. Miller and L.~Younes}, {\em Group actions, homeomorphisms, and
  matching: A general framework}, International Journal of Computer Vision, 41
  (2001), pp.~61--84.

\bibitem{WeakUnbddNet}
{\sc P.~Miller, A.~B{\"e}rd{\"e}llima, and M.~Wardetzky}, {\em Weak topologies
  for unbounded nets in {C}{A}{T}(0) spacesn}, 2022.

\bibitem{MumfordPattern}
{\sc D.~Mumford}, {\em Pattern theory: a unifying perspective}, in First
  {E}uropean {C}ongress of {M}athematics, {V}ol.\ {I} ({P}aris, 1992), vol.~119
  of Progr. Math., Birkh\"auser, Basel, 1994, pp.~187--224.

\bibitem{MumfordPatternTheory}
\leavevmode\vrule height 2pt depth -1.6pt width 23pt, {\em Pattern theory: a
  unifying perspective}, in Fields {M}edallists' lectures, vol.~5 of World Sci.
  Ser. 20th Century Math., World Sci. Publ., River Edge, NJ, 1997,
  pp.~226--261.

\bibitem{Navas}
{\sc A.~Navas}, {\em An {$L^1$} ergodic theorem with values in a non-positively
  curved space via a canonical barycenter map}, Ergodic Theory Dynam. Systems,
  33 (2013), pp.~609--623.

\bibitem{WassersteinAlignment}
{\sc S.~Pal, B.~Sen, and T.-K.~L. Wong}, {\em On the wasserstein alignment
  problem}, 2025.
\newblock Pre-print available at https://arxiv.org/abs/2503.06838.

\bibitem{Parna1}
{\sc K.~P\"{a}rna}, {\em Strong consistency of {$k$}-means clustering criterion
  in separable metric spaces}, Tartu Riikl. \"{U}l. Toimetised,  (1986),
  pp.~86--96.

\bibitem{Parna2}
\leavevmode\vrule height 2pt depth -1.6pt width 23pt, {\em On the stability of
  {$k$}-means clustering in metric spaces}, Tartu Riikl. \"{U}l. Toimetised,
  (1988), pp.~19--36.

\bibitem{FrechetRegression}
{\sc A.~Petersen and H.-G. M\"uller}, {\em Fr\'echet regression for random
  objects with {E}uclidean predictors}, Ann. Statist., 47 (2019), pp.~691--719.

\bibitem{Picard}
{\sc J.~Picard}, {\em Barycentres et martingales sur une vari\'{e}t\'{e}}, Ann.
  Inst. H. Poincar\'{e} Probab. Statist., 30 (1994), pp.~647--702.

\bibitem{PratelliRigo}
{\sc L.~Pratelli and P.~Rigo}, {\em A strong version of the {S}korohod
  representation theorem}, J. Theoret. Probab., 36 (2023), pp.~372--389.

\bibitem{WassersteinTomography}
{\sc R.~Rao, A.~Moscovich, and A.~Singer}, {\em Wasserstein k-means for
  clustering tomographic projections}, CoRR, abs/2010.09989 (2020).

\bibitem{PanaretosSantoroBW}
{\sc L.~V. Santoro and V.~M. Panaretos}, {\em Large sample theory for
  {B}ures-{W}asserstein barycenters}, 2023.
\newblock Pre-print available at https://arxiv.org/abs/2305.15592.

\bibitem{Schoetz1}
{\sc C.~Sch\"{o}tz}, {\em Convergence rates for the generalized {F}r\'{e}chet
  mean via the quadruple inequality}, Electron. J. Stat., 13 (2019),
  pp.~4280--4345.

\bibitem{Schoetz2}
{\sc C.~Sch\"otz}, {\em Strong laws of large numbers for generalizations of
  {F}r\'echet mean sets}, Statistics, 56 (2022), pp.~34--52.

\bibitem{SrivastavaKlassen}
{\sc A.~Srivastava and E.~P. Klassen}, {\em Functional and shape data
  analysis}, Springer Series in Statistics, Springer-Verlag, New York, 2016.

\bibitem{Struble1974}
{\sc R.~A. Struble}, {\em Metrics in locally compact groups}, Compositio
  Mathematica, 28 (1974), pp.~217--222.

\bibitem{SturmNPC}
{\sc K.-T. Sturm}, {\em Probability measures on metric spaces of nonpositive
  curvature}, in Heat kernels and analysis on manifolds, graphs, and metric
  spaces ({P}aris, 2002), vol.~338 of Contemp. Math., Amer. Math. Soc.,
  Providence, RI, 2003, pp.~357--390.

\bibitem{Sundaramoorthi}
{\sc G.~Sundaramoorthi, A.~Mennucci, S.~Soatto, and A.~Yezzi}, {\em A new
  geometric metric in the space of curves, and applications to tracking
  deforming objects by prediction and filtering}, SIAM J. Imaging Sci., 4
  (2011), pp.~109--145.

\bibitem{SverdrupThygeson}
{\sc H.~Sverdrup-Thygeson}, {\em Strong law of large numbers for measures of
  central tendency and dispersion of random variables in compact metric
  spaces}, Ann. Statist., 9 (1981), pp.~141--145.

\bibitem{Thorpe}
{\sc M.~Thorpe, F.~Theil, A.~M. Johansen, and N.~Cade}, {\em Convergence of the
  {$k$}-means minimization problem using {$\Gamma$}-convergence}, SIAM J. Appl.
  Math., 75 (2015), pp.~2444--2474.

\bibitem{ProcrustesWasserstein}
{\sc K.~M. Toukam}, {\em Procrustes wasserstein metric: A modified
  benamou-brenier approach with applications to latent gaussian distributions},
  2025.
\newblock Pre-print available at https://arxiv.org/abs/2503.16580.

\bibitem{TumpachPreston}
{\sc A.~B. Tumpach and S.~C. Preston}, {\em Quotient elastic metrics on the
  manifold of arc-length parameterized plane curves}, J. Geom. Mech., 9 (2017),
  pp.~227--256.

\bibitem{TurnerMedians}
{\sc K.~Turner}, {\em Medians of populations of persistence diagrams}, Homology
  Homotopy Appl., 22 (2020), pp.~255--282.

\bibitem{FrechetPersistence}
{\sc K.~Turner, Y.~Mileyko, S.~Mukherjee, and J.~Harer}, {\em Fr\'echet means
  for distributions of persistence diagrams}, Discrete Comput. Geom., 52
  (2014), pp.~44--70.

\bibitem{Varadarajan}
{\sc V.~S. Varadarajan}, {\em On the convergence of sample probability
  distributions}, Sankhy\={a}, 19 (1958), pp.~23--26.

\bibitem{Varadarajan2}
\leavevmode\vrule height 2pt depth -1.6pt width 23pt, {\em Weak convergence of
  measures on separable metric spaces}, Sankhy\=a, 19 (1958), pp.~15--22.

\bibitem{Villani}
{\sc C.~Villani}, {\em Optimal Transport}, vol.~338 of Grundlehren der
  Mathematischen Wissenschaften [Fundamental Principles of Mathematical
  Sciences], Springer-Verlag, Berlin, 2009.
\newblock Old and new.

\bibitem{SanovWasserstein}
{\sc R.~Wang, X.~Wang, and L.~Wu}, {\em Sanov's theorem in the {W}asserstein
  distance: a necessary and sufficient condition}, Statist. Probab. Lett., 80
  (2010), pp.~505--512.

\bibitem{YounesShapes}
{\sc L.~Younes}, {\em Shapes and diffeomorphisms}, vol.~171 of Applied
  Mathematical Sciences, Springer, Berlin, second~ed., 2019.

\bibitem{ZemelPanaretosProcrustes}
{\sc Y.~Zemel and V.~M. Panaretos}, {\em Fr\'echet means and {P}rocrustes
  analysis in {W}asserstein space}, Bernoulli, 25 (2019), pp.~932--976.

\bibitem{Ziezold}
{\sc H.~Ziezold}, {\em On expected figures and a strong law of large numbers
  for random elements in quasi-metric spaces}, in Transactions of the {S}eventh
  {P}rague {C}onference on {I}nformation {T}heory, {S}tatistical {D}ecision
  {F}unctions, {R}andom {P}rocesses and of the {E}ighth {E}uropean {M}eeting of
  {S}tatisticians ({T}ech. {U}niv. {P}rague, {P}rague, 1974),{V}ol. {A}, 1977,
  pp.~591--602.

\end{thebibliography}
	\bibliographystyle{siam}
	
\end{document}